\documentclass[12pt]{amsart}
\usepackage{amsfonts, amssymb, latexsym, hyperref, mathrsfs}
\usepackage[usenames, dvipsnames]{xcolor}
\usepackage{ytableau}
\usepackage{pgf,tikz}
\usetikzlibrary{calc, shapes, backgrounds,arrows,positioning}
\tikzset{>=stealth',
  head/.style = {fill = white, text=black}, 
  pil/.style={->,thick},
  und/.style={thick},
  junct/.style = {draw,circle,inner sep=0.5pt,outer sep=0pt, fill=black}
  }
\usepackage{enumerate}
\renewcommand{\familydefault}{ppl}
\newtheorem{theorem}{Theorem}[section]
\newtheorem{proposition}[theorem]{Proposition}
\newtheorem{lemma}[theorem]{Lemma}

\newtheorem{claim}[theorem]{Claim}
\newtheorem*{claim*}{Claim}

\newtheorem{conjecture}[theorem]{Conjecture}

\theoremstyle{remark}
\newtheorem{example}[theorem]{Example}

\numberwithin{equation}{section}

\newcommand\complexes{{\mathbb C}}

\DeclareMathOperator{\GL}{\sf GL}

\newcommand{\blank}{\phantom{2}}
\newcommand{\DD}{\mathcal{D}}
\newcommand{\EE}{\mathcal{E}}
\newcommand{\FF}{\mathcal{F}}
\newcommand{\GG}{\mathcal{G}}
\newcommand{\HH}{\mathcal{H}}
\newcommand{\QQ}{\mathcal{Q}}
\newcommand{\LL}{\mathcal{L}}
\newcommand{\w}{{\sf w}}
\newcommand\x{{\sf x}}
\newcommand\y{{\sf y}}
\newcommand\z{{\sf z}}
\newcommand\uuu{{\sf u}}
\newcommand\aaa{{\sf a}}
\newcommand\bbb{{\sf b}}
\newcommand{\m}{{\sf mswap}}
\newcommand{\revm}{{\sf revmswap}}
\DeclareMathOperator{\wt}{{\tt wt}}
\DeclareMathOperator{\lab}{{\tt label}}
\DeclareMathOperator{\family}{{\tt family}}
\DeclareMathOperator{\im}{{\rm im}}
\newcommand{\gap}{\hspace{1in} \\ \vspace{-.2in}}
\newcommand{\upper}[1]{#1_\star}

\newcommand*{\Scale}[2][4]{\scalebox{#1}{$#2$}}%
\newcommand{\excise}[1]{}%{$\star$\textsc{#1}$\star$}
\newcommand{\comment}[1]{$\star${\sf\textbf{#1}}$\star$}

\newcommand*\circled[1]{\tikz[baseline=(char.base)]{\node[shape=circle,draw,inner sep=1pt] (char) {$#1$};}}
\definecolor{qqffqq}{rgb}{0,0.87,0} %green
\definecolor{ttfftt}{rgb}{0,0.87,0} %green
\definecolor{zzqqzz}{rgb}{0.5,0,0.5} %purple
\definecolor{ccwwff}{rgb}{0.5,0,0.5} %purple
\definecolor{cccccc}{rgb}{0.67,0.67,0.67} %grey
\definecolor{eqeqeq}{rgb}{0.65,0.65,0.65} %grey
\definecolor{zzttqq}{rgb}{1,1,0} %orange

\begin{document}
\pagestyle{plain}

\mbox{}
\title{Equivariant $K$-theory of Grassmannians II: \\ the Knutson-Vakil conjecture}
\author{Oliver Pechenik}
\author{Alexander Yong}
\address{Dept.~of Mathematics, U.~Illinois at
Urbana-Champaign, Urbana, IL 61801, USA}
\email{pecheni2@illinois.edu, ayong@uiuc.edu}

\date{August 3, 2015}

\begin{abstract}
In 2005, A.~Knutson--R.~Vakil conjectured a \emph{puzzle} rule for equivariant $K$-theory of Grassmannians.
We resolve this conjecture. After giving a correction, we establish a modified rule by combinatorially
connecting it to the authors' recently proved tableau rule for the same Schubert calculus problem.
\end{abstract}

\maketitle

\tableofcontents

\section{Introduction}
\label{sec:intro}

A.~Knutson--R.~Vakil 
\cite[$\mathsection$5]{Coskun.Vakil}
 conjectured a combinatorial rule for the structure coefficients of the torus-equivariant $K$-theory ring
of a Grassmannian. The structure coefficients are with respect
to the basis of Schubert structure sheaves. Their rule extends 
\emph{puzzles}, combinatorial objects founded in work of A.~Knutson-T.~Tao \cite{Knutson.Tao} and in their collaboration with C.~Woodward \cite{KTW}. The various puzzle rules play a prominent role in modern Schubert calculus, see e.g., \cite{Buch.Kresch.Tamvakis, Vakil:annals, Coskun.Vakil},
recent developments \cite{Knutson:positroid, Knutson.Purbhoo, Buch.Kresch.Purbhoo.Tamvakis, Buch:quantum} and the references therein.

This paper is a sequel to \cite{PeYo} where we gave the first proved tableau rules for these structure coefficients,
including a conjecture of H.~Thomas and the second author \cite{Thomas.Yong:H_T}. Here we use these results to prove a mild correction of the puzzle conjecture.

\subsection{The puzzle conjecture}
\label{sec:puzzle_conjecture}
Let $X = {\rm Gr}_k(\complexes^n)$ denote the Grassmannian of $k$-dimensional subspaces of $\complexes^n$. The general linear group ${\sf GL}_n$ acts transitively on $X$ by change of basis. The Borel subgroup ${\sf B}\subset 
{\sf GL_n}$ of invertible lower triangular matrices acts on $X$ with finitely many orbits, i.e., the 
{\bf Schubert cells} $X_\lambda^\circ$. These orbits are indexed by $\{0,1\}$-sequences $\lambda$ of length $n$
with $k$-many $1$'s. The {\bf Schubert varieties} are the Zariski closures $X_\lambda:=\overline{X_\lambda^\circ}$.
The $X_\lambda$ are stable under the action of the maximal torus
${\sf T}\subset {\sf B}$ of invertible diagonal matrices. Therefore their structure sheaves ${\mathcal O}_{X_\lambda}$ admit classes in $K_{\sf T}(X)$, the Grothendieck ring of {\sf T}-equivariant vector bundles over $X$. Now, $K_{\sf T}(X)$ is a 
$K_{\sf T}({\rm pt})$-module and the ${n\choose k}$ Schubert classes form a module basis. One may make a standard
identification $K_{\sf T}({\rm pt})\cong {\mathbb Z}[1-\frac{t_i}{t_{i+1}}:1\leq i<n]$. The structure coefficients
$K_{\lambda,\mu}^{\nu}\in K_{\sf T}({\rm pt})$ are defined by
\[[{\mathcal O}_{X_\lambda}]\cdot [{\mathcal O}_{X_\mu}]=\sum_{\nu}K_{\lambda,\mu}^\nu[{\mathcal O}_{X_{\nu}}].\]
Consider the $n$-length equilateral triangle oriented as $\Delta$. A {\bf puzzle}
is a filling of $\Delta$ with the following {\bf puzzle pieces}:
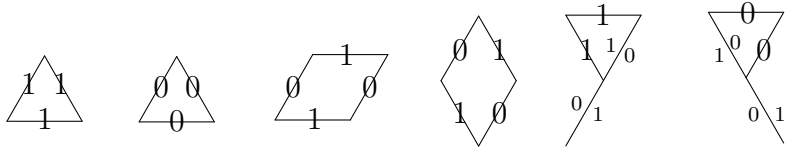
\[\begin{picture}(350,50)
%Black triangle
\put(0,0){
\begin{tikzpicture}[line cap=round,line join=round,>=triangle 45,x=1.0cm,y=1.0cm]
\clip(0.7,-0.38) rectangle (2.22,1.08);
%\fill[fill=black,fill opacity=0.0] (1,0) -- (2,0) -- (1.5,0.87) -- cycle;
\draw (1,0)-- (2,0);
\draw (2,0)-- (1.5,0.87);
\draw (1.5,0.87)-- (1,0);
\end{tikzpicture}}
\put(17,21){$1$}
\put(29,21){$1$}
\put(23,8){$1$}

%White triangle
\put(50,4){
\begin{tikzpicture}[line cap=round,line join=round,>=triangle 45,x=1.0cm,y=1.0cm]
\clip(0.7,-0.23) rectangle (2.26,0.97);
\draw (1,0)-- (2,0);
\draw (2,0)-- (1.5,0.87);
\draw (1.5,0.87)-- (1,0);
\end{tikzpicture}}
\put(67,20){$0$}
\put(79,20){$0$}
\put(73,7){$0$}

%Grey rhombus
\put(100,1){
\begin{tikzpicture}[line cap=round,line join=round,>=triangle 45,x=1.0cm,y=1.0cm]
\clip(0.65,-0.36) rectangle (2.86,1.08);
%\fill[black,fill=cccccc,fill opacity=0.0] (1,0) -- (1.5,0.87) -- (2.5,0.87) -- (2,0) -- cycle;
\draw  (1,0)-- (1.5,0.87);
\draw (1.5,0.87)-- (2.5,0.87);
\draw  (2.5,0.87)-- (2,0);
\draw (2,0)-- (1,0);
\end{tikzpicture}
}
\put(117,20){$0$}
\put(146,20){$0$}
\put(125,8){$1$}
\put(137,32){$1$}

%Green rhombus
\put(165,-10){\begin{tikzpicture}[line cap=round,line join=round,>=triangle 45,x=1.0cm,y=1.0cm]
\clip(1.12,-0.4) rectangle (2.76,2.1);
%\fill[color=qqffqq,fill=qqffqq,fill opacity=0.0] (2,1.73) -- (1.5,0.87) -- (2,0) -- (2.5,0.87) -- cycle;
\draw (2,1.73)-- (1.5,0.87);
\draw (1.5,0.87)-- (2,0);
\draw  (2,0)-- (2.5,0.87);
\draw (2.5,0.87)-- (2,1.73);
\end{tikzpicture}}
\put(180,34){$0$}
\put(195,34){$1$}
\put(180,10){$1$}
\put(195,10){$0$}

%Orange gashed triangle
\put(210,-4){
\begin{tikzpicture}[line cap=round,line join=round,>=triangle 45,x=1.0cm,y=1.0cm]
\clip(1.15,-1.06) rectangle (2.76,0.97);
%\fill[color=zzttqq,fill=zzttqq,fill opacity=0.0] (1.5,0.87) -- (2,0) -- (2.5,0.87) -- cycle;
\draw (1.5,0.87)-- (2,0);
\draw  (2,0)-- (2.5,0.87);
\draw (2.5,0.87)-- (1.5,0.87);
\draw  (2,0)-- (1.5,-0.87);
\end{tikzpicture}}
\put(228,34){$1$}
\put(234,47){$1$}
\put(238,37){\tiny{$1$}}
\put(245,33){\tiny{$0$}}
\put(225,15){\tiny{$0$}}
\put(233,11){\tiny{$1$}}

%Purple gashed triangle
\put(260,-4){
\begin{tikzpicture}[line cap=round,line join=round,>=triangle 45,x=1.0cm,y=1.0cm]
\clip(1.01,-1.1) rectangle (2.98,1.19);
%\fill[color=zzqqzz,fill=zzqqzz,fill opacity=0.0] (1.5,0.87) -- (2,0) -- (2.5,0.87) -- cycle;
\draw  (1.5,0.87)-- (2,0);
\draw (2,0)-- (2.5,0.87);
\draw (2.5,0.87)-- (1.5,0.87);
\draw (2.5,-0.87)-- (2,0);
\end{tikzpicture}
}
\put(279,33){\tiny{$1$}}
\put(285,38){\tiny{$0$}}
\put(292,11){\tiny{$0$}}
\put(302,11){\tiny{$1$}}
\put(295,34){$0$}
\put(289,48){$0$}
\end{picture}\]
The double-labeled edges are {\bf gashed}.
A {\bf filling} requires that 
the common (non-gashed) edges of adjacent puzzle pieces share the same label. Two gashed edges may not be overlayed. The pieces on either side of a gash must have the indicated labels.
The first three may be rotated but the fourth
({\bf equivariant piece}) may not \cite{Knutson.Tao}. We call the remainder {\bf KV-pieces}; these may not be rotated.
The fifth piece may only be placed if the equivariant piece is attached to its left. 
There is a ``nonlocal'' requirement \cite[$\mathsection$5]{Coskun.Vakil} for using the sixth 
piece: it ``may only be placed (when completing the puzzle from top to bottom and
left to right as usual) if the edges to its right are a (possibly empty) series of horizontal $0$'s followed by a $1$.'' A {\bf KV-puzzle} is a puzzle filling of $\Delta$. 

Let $\Delta_{\lambda,\mu,\nu}$ be $\Delta$ with the boundary given by 
\begin{itemize}
\item $\lambda$ as read $\nearrow$ along the left side;
\item $\mu$ as read $\searrow$ along the right side; and
\item $\nu$ as read $\rightarrow$ along the bottom side.
\end{itemize}
The {\bf weight} ${\rm wt}(P)$ of a {\bf KV-puzzle} $P$ is a product of the following factors. Each KV-piece contributes a factor of $-1$. For each equivariant piece one draws a
$\searrow$ diagonal arrow
from the center of the piece to the $\nu$-side of $\Delta$; let $a$ be the 
unit segment of the $\nu$-boundary, as counted from the right. Similarly one determines 
$b$ by drawing a 
$\swarrow$ antidiagonal arrow. The equivariant piece contributes a factor of $1-\frac{t_a}{t_b}$.

\begin{conjecture}[The Knutson-Vakil puzzle conjecture]
\label{conj:Knutson.Vakil}
$K_{\lambda,\mu}^{\nu}=\sum_{P} {\rm wt}(P)$ where the sum is over all KV-puzzles of 
 $\Delta_{\lambda,\mu,\nu}$.
\end{conjecture}

We consider the structure coefficient $K_{01001,00101}^{10010}$  for ${\rm Gr}_2({\mathbb C}^5)$. The reader can check that there are six KV-puzzles $P_1,P_2,\ldots,P_6$ with the indicated weights. Henceforth, we color-code the six puzzle pieces black, white, grey, green, yellow and purple, respectively.

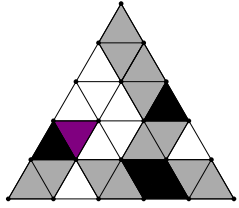
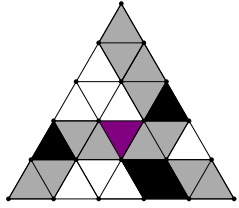
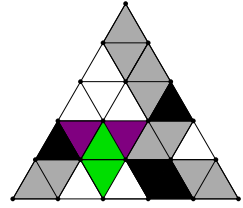
\begin{picture}(240,113)
\put(50,10){${\rm wt}(P_1)=-1$}
\put(0,15){\Scale[0.6]{
\begin{tikzpicture}[line cap=round,line join=round,>=triangle 45,x=1.0cm,y=1.0cm]
\clip(-2.47,-1.02) rectangle (7.63,5.36);
\begin{tiny}
\fill[color=black,fill=cccccc,fill opacity=1.0] (0,0) -- (0.5,0.87) -- (1.5,0.87) -- (1,0) -- cycle;
\fill[color=black,fill=black,fill opacity=1.0] (0.5,0.87) -- (1.5,0.87) -- (1,1.73) -- cycle;
\fill[color=black,fill=cccccc,fill opacity=1.0] (2.5,4.33) -- (3,3.46) -- (2.5,2.6) -- (2,3.46) -- cycle;
\fill[color=black,fill=cccccc,fill opacity=1.0] (3,3.46) -- (3.5,2.6) -- (3,1.73) -- (2.5,2.6) -- cycle;
\fill[color=black,fill=black,fill opacity=1.0] (3.5,2.6) -- (3,1.73) -- (4,1.73) -- cycle;
\fill[color=black,fill=cccccc,fill opacity=1.0] (3.5,0.87) -- (4.5,0.87) -- (5,0) -- (4,0) -- cycle;
\fill[color=black,fill=black,fill opacity=1.0] (3.5,0.87) -- (3,0) -- (4,0) -- cycle;
\fill[color=black,fill=cccccc,fill opacity=1.0] (3,1.73) -- (2.5,0.87) -- (3.5,0.87) -- (4,1.73) -- cycle;
\fill[color=black,fill=black,fill opacity=1.0] (2.5,0.87) -- (3.5,0.87) -- (3,0) -- cycle;
\fill[color=black,fill=cccccc,fill opacity=1.0] (1.5,0.87) -- (2.5,0.87) -- (3,0) -- (2,0) -- cycle;
\fill[color=black,fill=zzqqzz,fill opacity=1.0] (1,1.73) -- (2,1.73) -- (1.5,0.87) -- cycle;
\draw [color=black] (0,0)-- (5,0);
\draw [color=black] (5,0)-- (2.5,4.33);
\draw [color=black] (2.5,4.33)-- (0,0);
\draw [color=black] (0.5,0.87)-- (1,0);
\draw [color=black] (1,1.73)-- (2,0);
\draw [color=black] (1.5,2.6)-- (3,0);
\draw [color=black] (2,3.46)-- (4,0);
\draw [color=black] (0.5,0.87)-- (4.5,0.87);
\draw [color=black] (1,1.73)-- (4,1.73);
\draw [color=black] (1.5,2.6)-- (3.5,2.6);
\draw [color=black] (2,3.46)-- (3,3.46);
\draw [color=black] (1,0)-- (3,3.46);
\draw [color=black] (3.5,2.6)-- (2,0);
\draw [color=black] (3,0)-- (4,1.73);
\draw [color=black] (4.5,0.87)-- (4,0);
\draw [color=black] (0,0)-- (0.5,0.87);
\draw [color=black] (0.5,0.87)-- (1.5,0.87);
\draw [color=black] (1.5,0.87)-- (1,0);
\draw [color=black] (1,0)-- (0,0);
\draw [color=black] (0.5,0.87)-- (1.5,0.87);
\draw [color=black] (1.5,0.87)-- (1,1.73);
\draw [color=black] (1,1.73)-- (0.5,0.87);
\draw [color=black] (2.5,4.33)-- (3,3.46);
\draw [color=black] (3,3.46)-- (2.5,2.6);
\draw [color=black] (2.5,2.6)-- (2,3.46);
\draw [color=black] (2,3.46)-- (2.5,4.33);
\draw [color=black] (3,3.46)-- (3.5,2.6);
\draw [color=black] (3.5,2.6)-- (3,1.73);
\draw [color=black] (3,1.73)-- (2.5,2.6);
\draw [color=black] (2.5,2.6)-- (3,3.46);
\draw [color=black] (3.5,2.6)-- (3,1.73);
\draw [color=black] (3,1.73)-- (4,1.73);
\draw [color=black] (4,1.73)-- (3.5,2.6);
\draw [color=black] (3.5,0.87)-- (4.5,0.87);
\draw [color=black] (4.5,0.87)-- (5,0);
\draw [color=black] (5,0)-- (4,0);
\draw [color=black] (4,0)-- (3.5,0.87);
\draw [color=black] (3.5,0.87)-- (3,0);
\draw [color=black] (3,0)-- (4,0);
\draw [color=black] (4,0)-- (3.5,0.87);
\draw [color=black] (3,1.73)-- (2.5,0.87);
\draw [color=black] (2.5,0.87)-- (3.5,0.87);
\draw [color=black] (3.5,0.87)-- (4,1.73);
\draw [color=black] (4,1.73)-- (3,1.73);
\draw [color=black] (2.5,0.87)-- (3.5,0.87);
\draw [color=black] (3.5,0.87)-- (3,0);
\draw [color=black] (3,0)-- (2.5,0.87);
\draw [color=black] (1.5,0.87)-- (2.5,0.87);
\draw [color=black] (2.5,0.87)-- (3,0);
\draw [color=black] (3,0)-- (2,0);
\draw [color=black] (2,0)-- (1.5,0.87);
\draw [color=black] (1,1.73)-- (2,1.73);
\draw [color=black] (2,1.73)-- (1.5,0.87);
\draw [color=black] (1.5,0.87)-- (1,1.73);
\fill [color=black] (0,0) circle (1.5pt);
\fill [color=black] (5,0) circle (1.5pt);
\fill [color=black] (2.5,4.33) circle (1.5pt);
\fill [color=black] (1,0) circle (1.5pt);
\fill [color=black] (2,0) circle (1.5pt);
\fill [color=black] (3,0) circle (1.5pt);
\fill [color=black] (4,0) circle (1.5pt);
\fill [color=black] (4.5,0.87) circle (1.5pt);
\fill [color=black] (4,1.73) circle (1.5pt);
\fill [color=black] (3.5,2.6) circle (1.5pt);
\fill [color=black] (3,3.46) circle (1.5pt);
\fill [color=black] (0.5,0.87) circle (1.5pt);
\fill [color=black] (1,1.73) circle (1.5pt);
\fill [color=black] (1.5,2.6) circle (1.5pt);
\fill [color=black] (2,3.46) circle (1.5pt);
\fill [color=black] (1.5,0.87) circle (1.5pt);
\fill [color=black] (2.5,2.6) circle (1.5pt);
\fill [color=black] (3,1.73) circle (1.5pt);
\fill [color=black] (3.5,0.87) circle (1.5pt);
\fill [color=black] (2.5,0.87) circle (1.5pt);
\fill [color=black] (2,1.73) circle (1.5pt);
\end{tiny}
\end{tikzpicture}}}

\put(170,10){${\rm wt}(P_2)=-1$}
\put(120,15){\Scale[0.6]{
\begin{tikzpicture}[line cap=round,line join=round,>=triangle 45,x=1.0cm,y=1.0cm]
\clip(-2.47,-1.02) rectangle (7.63,5.36);
\begin{tiny}
\fill[color=cccccc,fill=cccccc,fill opacity=1.0] (0,0) -- (0.5,0.87) -- (1.5,0.87) -- (1,0) -- cycle;
\fill[color=black,fill=black,fill opacity=1.0] (0.5,0.87) -- (1.5,0.87) -- (1,1.73) -- cycle;
\fill[color=cccccc,fill=cccccc,fill opacity=1.0] (2.5,4.33) -- (3,3.46) -- (2.5,2.6) -- (2,3.46) -- cycle;
\fill[color=cccccc,fill=cccccc,fill opacity=1.0] (3,3.46) -- (3.5,2.6) -- (3,1.73) -- (2.5,2.6) -- cycle;
\fill[color=black,fill=black,fill opacity=1.0] (3.5,2.6) -- (3,1.73) -- (4,1.73) -- cycle;
\fill[color=cccccc,fill=cccccc,fill opacity=1.0] (3.5,0.87) -- (4.5,0.87) -- (5,0) -- (4,0) -- cycle;
\fill[color=black,fill=black,fill opacity=1.0] (3.5,0.87) -- (3,0) -- (4,0) -- cycle;
\fill[color=cccccc,fill=cccccc,fill opacity=1.0] (3,1.73) -- (2.5,0.87) -- (3.5,0.87) -- (4,1.73) -- cycle;
\fill[color=black,fill=black,fill opacity=1.0] (2.5,0.87) -- (3.5,0.87) -- (3,0) -- cycle;
\fill[color=zzqqzz,fill=zzqqzz,fill opacity=1.0] (3,1.73) -- (2.5,0.87) -- (2,1.73) -- cycle;
\fill[color=cccccc,fill=cccccc,fill opacity=1.0] (1,1.73) -- (2,1.73) -- (2.5,0.87) -- (1.5,0.87) -- cycle;
\draw [color=black] (0,0)-- (5,0);
\draw [color=black] (5,0)-- (2.5,4.33);
\draw [color=black] (2.5,4.33)-- (0,0);
\draw [color=black] (0.5,0.87)-- (1,0);
\draw [color=black] (1,1.73)-- (2,0);
\draw [color=black] (1.5,2.6)-- (3,0);
\draw [color=black] (2,3.46)-- (4,0);
\draw [color=black] (0.5,0.87)-- (4.5,0.87);
\draw [color=black] (1,1.73)-- (4,1.73);
\draw [color=black] (1.5,2.6)-- (3.5,2.6);
\draw [color=black] (2,3.46)-- (3,3.46);
\draw [color=black] (1,0)-- (3,3.46);
\draw [color=black] (3.5,2.6)-- (2,0);
\draw [color=black] (3,0)-- (4,1.73);
\draw [color=black] (4.5,0.87)-- (4,0);
\draw [color=black] (0,0)-- (0.5,0.87);
\draw [color=black] (0.5,0.87)-- (1.5,0.87);
\draw [color=black] (1.5,0.87)-- (1,0);
\draw [color=black] (1,0)-- (0,0);
\draw [color=black] (0.5,0.87)-- (1.5,0.87);
\draw [color=black] (1.5,0.87)-- (1,1.73);
\draw [color=black] (1,1.73)-- (0.5,0.87);
\draw [color=black] (2.5,4.33)-- (3,3.46);
\draw [color=black] (3,3.46)-- (2.5,2.6);
\draw [color=black] (2.5,2.6)-- (2,3.46);
\draw [color=black] (2,3.46)-- (2.5,4.33);
\draw [color=black] (3,3.46)-- (3.5,2.6);
\draw [color=black] (3.5,2.6)-- (3,1.73);
\draw [color=black] (3,1.73)-- (2.5,2.6);
\draw [color=black] (2.5,2.6)-- (3,3.46);
\draw [color=black] (3.5,2.6)-- (3,1.73);
\draw [color=black] (3,1.73)-- (4,1.73);
\draw [color=black] (4,1.73)-- (3.5,2.6);
\draw [color=black] (3.5,0.87)-- (4.5,0.87);
\draw [color=black] (4.5,0.87)-- (5,0);
\draw [color=black] (5,0)-- (4,0);
\draw [color=black] (4,0)-- (3.5,0.87);
\draw [color=black] (3.5,0.87)-- (3,0);
\draw [color=black] (3,0)-- (4,0);
\draw [color=black] (4,0)-- (3.5,0.87);
\draw [color=black] (3,1.73)-- (2.5,0.87);
\draw [color=black] (2.5,0.87)-- (3.5,0.87);
\draw [color=black] (3.5,0.87)-- (4,1.73);
\draw [color=black] (4,1.73)-- (3,1.73);
\draw [color=black] (2.5,0.87)-- (3.5,0.87);
\draw [color=black] (3.5,0.87)-- (3,0);
\draw [color=black] (3,0)-- (2.5,0.87);
\draw [color=black] (1.5,0.87)-- (2.5,0.87);
\draw [color=black] (2.5,0.87)-- (3,0);
\draw [color=black] (3,0)-- (2,0);
\draw [color=black] (2,0)-- (1.5,0.87);
\draw [color=black] (1,1.73)-- (2,1.73);
\draw [color=black] (2,1.73)-- (1.5,0.87);
\draw [color=black] (1.5,0.87)-- (1,1.73);
\fill [color=black] (0,0) circle (1.5pt);
\fill [color=black] (5,0) circle (1.5pt);
\fill [color=black] (2.5,4.33) circle (1.5pt);
\fill [color=black] (1,0) circle (1.5pt);
\fill [color=black] (2,0) circle (1.5pt);
\fill [color=black] (3,0) circle (1.5pt);
\fill [color=black] (4,0) circle (1.5pt);
\fill [color=black] (4.5,0.87) circle (1.5pt);
\fill [color=black] (4,1.73) circle (1.5pt);
\fill [color=black] (3.5,2.6) circle (1.5pt);
\fill [color=black] (3,3.46) circle (1.5pt);
\fill [color=black] (0.5,0.87) circle (1.5pt);
\fill [color=black] (1,1.73) circle (1.5pt);
\fill [color=black] (1.5,2.6) circle (1.5pt);
\fill [color=black] (2,3.46) circle (1.5pt);
\fill [color=black] (1.5,0.87) circle (1.5pt);
\fill [color=black] (2.5,2.6) circle (1.5pt);
\fill [color=black] (3,1.73) circle (1.5pt);
\fill [color=black] (3.5,0.87) circle (1.5pt);
\fill [color=black] (2.5,0.87) circle (1.5pt);
\fill [color=black] (2,1.73) circle (1.5pt);
\end{tiny}
\end{tikzpicture}}}

\put(270,10){${\rm wt}(P_3)=(-1)^2(1-\frac{t_3}{t_4})$}
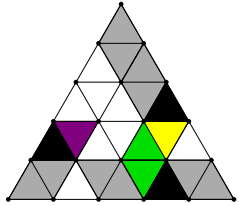
\put(240,15){\Scale[0.6]{
\begin{tikzpicture}[line cap=round,line join=round,>=triangle 45,x=1.0cm,y=1.0cm]
\clip(-2.47,-1.02) rectangle (7.63,5.36);
\begin{tiny}
\fill[color=cccccc,fill=cccccc,fill opacity=1.0] (0,0) -- (0.5,0.87) -- (1.5,0.87) -- (1,0) -- cycle;
\fill[color=black,fill=black,fill opacity=1.0] (0.5,0.87) -- (1.5,0.87) -- (1,1.73) -- cycle;
\fill[color=cccccc,fill=cccccc,fill opacity=1.0] (2.5,4.33) -- (3,3.46) -- (2.5,2.6) -- (2,3.46) -- cycle;
\fill[color=cccccc,fill=cccccc,fill opacity=1.0] (3,3.46) -- (3.5,2.6) -- (3,1.73) -- (2.5,2.6) -- cycle;
\fill[color=black,fill=black,fill opacity=1.0] (3.5,2.6) -- (3,1.73) -- (4,1.73) -- cycle;
\fill[color=cccccc,fill=cccccc,fill opacity=1.0] (3.5,0.87) -- (4.5,0.87) -- (5,0) -- (4,0) -- cycle;
\fill[color=black,fill=black,fill opacity=1.0] (3.5,0.87) -- (3,0) -- (4,0) -- cycle;
\fill[color=cccccc,fill=cccccc,fill opacity=1.0] (3,1.73) -- (2.5,0.87) -- (3.5,0.87) -- (4,1.73) -- cycle;
\fill[color=black,fill=black,fill opacity=1.0] (2.5,0.87) -- (3.5,0.87) -- (3,0) -- cycle;
\fill[color=zzqqzz,fill=zzqqzz,fill opacity=1.0] (3,1.73) -- (2.5,0.87) -- (2,1.73) -- cycle;
\fill[color=ttfftt,fill=ttfftt,fill opacity=1.0] (2,1.73) -- (1.5,0.87) -- (2,0) -- (2.5,0.87) -- cycle;
\fill[color=zzqqzz,fill=zzqqzz,fill opacity=1.0] (2,1.73) -- (1,1.73) -- (1.5,0.87) -- cycle;
\draw [color=black] (0,0)-- (5,0);
\draw [color=black] (5,0)-- (2.5,4.33);
\draw [color=black] (2.5,4.33)-- (0,0);
\draw [color=black] (0.5,0.87)-- (1,0);
\draw [color=black] (1,1.73)-- (2,0);
\draw [color=black] (1.5,2.6)-- (3,0);
\draw [color=black] (2,3.46)-- (4,0);
\draw [color=black] (0.5,0.87)-- (4.5,0.87);
\draw [color=black] (1,1.73)-- (4,1.73);
\draw [color=black] (1.5,2.6)-- (3.5,2.6);
\draw [color=black] (2,3.46)-- (3,3.46);
\draw [color=black] (1,0)-- (3,3.46);
\draw [color=black] (3.5,2.6)-- (2,0);
\draw [color=black] (3,0)-- (4,1.73);
\draw [color=black] (4.5,0.87)-- (4,0);
\draw [color=black] (0,0)-- (0.5,0.87);
\draw [color=black] (0.5,0.87)-- (1.5,0.87);
\draw [color=black] (1.5,0.87)-- (1,0);
\draw [color=black] (1,0)-- (0,0);
\draw [color=black] (0.5,0.87)-- (1.5,0.87);
\draw [color=black] (1.5,0.87)-- (1,1.73);
\draw [color=black] (1,1.73)-- (0.5,0.87);
\draw [color=black] (2.5,4.33)-- (3,3.46);
\draw [color=black] (3,3.46)-- (2.5,2.6);
\draw [color=black] (2.5,2.6)-- (2,3.46);
\draw [color=black] (2,3.46)-- (2.5,4.33);
\draw [color=black] (3,3.46)-- (3.5,2.6);
\draw [color=black] (3.5,2.6)-- (3,1.73);
\draw [color=black] (3,1.73)-- (2.5,2.6);
\draw [color=black] (2.5,2.6)-- (3,3.46);
\draw [color=black] (3.5,2.6)-- (3,1.73);
\draw [color=black] (3,1.73)-- (4,1.73);
\draw [color=black] (4,1.73)-- (3.5,2.6);
\draw [color=black] (3.5,0.87)-- (4.5,0.87);
\draw [color=black] (4.5,0.87)-- (5,0);
\draw [color=black] (5,0)-- (4,0);
\draw [color=black] (4,0)-- (3.5,0.87);
\draw [color=black] (3.5,0.87)-- (3,0);
\draw [color=black] (3,0)-- (4,0);
\draw [color=black] (4,0)-- (3.5,0.87);
\draw [color=black] (3,1.73)-- (2.5,0.87);
\draw [color=black] (2.5,0.87)-- (3.5,0.87);
\draw [color=black] (3.5,0.87)-- (4,1.73);
\draw [color=black] (4,1.73)-- (3,1.73);
\draw [color=black] (2.5,0.87)-- (3.5,0.87);
\draw [color=black] (3.5,0.87)-- (3,0);
\draw [color=black] (3,0)-- (2.5,0.87);
\draw [color=black] (1.5,0.87)-- (2.5,0.87);
\draw [color=black] (2.5,0.87)-- (3,0);
\draw [color=black] (3,0)-- (2,0);
\draw [color=black] (2,0)-- (1.5,0.87);
\draw [color=black] (1,1.73)-- (2,1.73);
\draw [color=black] (2,1.73)-- (1.5,0.87);
\draw [color=black] (1.5,0.87)-- (1,1.73);
\fill [color=black] (0,0) circle (1.5pt);
\fill [color=black] (5,0) circle (1.5pt);
\fill [color=black] (2.5,4.33) circle (1.5pt);
\fill [color=black] (1,0) circle (1.5pt);
\fill [color=black] (2,0) circle (1.5pt);
\fill [color=black] (3,0) circle (1.5pt);
\fill [color=black] (4,0) circle (1.5pt);
\fill [color=black] (4.5,0.87) circle (1.5pt);
\fill [color=black] (4,1.73) circle (1.5pt);
\fill [color=black] (3.5,2.6) circle (1.5pt);
\fill [color=black] (3,3.46) circle (1.5pt);
\fill [color=black] (0.5,0.87) circle (1.5pt);
\fill [color=black] (1,1.73) circle (1.5pt);
\fill [color=black] (1.5,2.6) circle (1.5pt);
\fill [color=black] (2,3.46) circle (1.5pt);
\fill [color=black] (1.5,0.87) circle (1.5pt);
\fill [color=black] (2.5,2.6) circle (1.5pt);
\fill [color=black] (3,1.73) circle (1.5pt);
\fill [color=black] (3.5,0.87) circle (1.5pt);
\fill [color=black] (2.5,0.87) circle (1.5pt);
\fill [color=black] (2,1.73) circle (1.5pt);
\end{tiny}
\end{tikzpicture}}}
\end{picture}

\begin{picture}(240,113)
\put(20,10){${\rm wt}(P_4)=(-1)^2(1-\frac{t_2}{t_3})$}
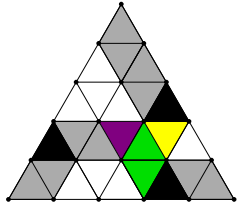
\put(0,15){\Scale[0.6]{
\begin{tikzpicture}[line cap=round,line join=round,>=triangle 45,x=1.0cm,y=1.0cm]
\clip(-2.47,-1.02) rectangle (7.63,5.36);
\begin{tiny}
\fill[color=cccccc,fill=cccccc,fill opacity=1.0] (0,0) -- (0.5,0.87) -- (1.5,0.87) -- (1,0) -- cycle;
\fill[color=black,fill=black,fill opacity=1.0] (0.5,0.87) -- (1.5,0.87) -- (1,1.73) -- cycle;
\fill[color=cccccc,fill=cccccc,fill opacity=1.0] (2.5,4.33) -- (3,3.46) -- (2.5,2.6) -- (2,3.46) -- cycle;
\fill[color=cccccc,fill=cccccc,fill opacity=1.0] (3,3.46) -- (3.5,2.6) -- (3,1.73) -- (2.5,2.6) -- cycle;
\fill[color=black,fill=black,fill opacity=1.0] (3.5,2.6) -- (3,1.73) -- (4,1.73) -- cycle;
\fill[color=cccccc,fill=cccccc,fill opacity=1.0] (3.5,0.87) -- (4.5,0.87) -- (5,0) -- (4,0) -- cycle;
\fill[color=black,fill=black,fill opacity=1.0] (3.5,0.87) -- (3,0) -- (4,0) -- cycle;
\fill[color=zzttqq,fill=zzttqq,fill opacity=1.0] (3,1.73) -- (4,1.73) -- (3.5,0.87) -- cycle;
\fill[color=ttfftt,fill=ttfftt,fill opacity=1.0] (3,1.73) -- (3.5,0.87) -- (3,0) -- (2.5,0.87) -- cycle;
\fill[color=cccccc,fill=cccccc,fill opacity=1.0] (1.5,0.87) -- (2.5,0.87) -- (3,0) -- (2,0) -- cycle;
\fill[color=zzqqzz,fill=zzqqzz,fill opacity=1.0] (1,1.73) -- (2,1.73) -- (1.5,0.87) -- cycle;
\draw [color=black] (0,0)-- (5,0);
\draw [color=black] (5,0)-- (2.5,4.33);
\draw [color=black] (2.5,4.33)-- (0,0);
\draw [color=black] (0.5,0.87)-- (1,0);
\draw [color=black] (1,1.73)-- (2,0);
\draw [color=black] (1.5,2.6)-- (3,0);
\draw [color=black] (2,3.46)-- (4,0);
\draw [color=black] (0.5,0.87)-- (4.5,0.87);
\draw [color=black] (1,1.73)-- (4,1.73);
\draw [color=black] (1.5,2.6)-- (3.5,2.6);
\draw [color=black] (2,3.46)-- (3,3.46);
\draw [color=black] (1,0)-- (3,3.46);
\draw [color=black] (3.5,2.6)-- (2,0);
\draw [color=black] (3,0)-- (4,1.73);
\draw [color=black] (4.5,0.87)-- (4,0);
\draw [color=black] (0,0)-- (0.5,0.87);
\draw [color=black] (0.5,0.87)-- (1.5,0.87);
\draw [color=black] (1.5,0.87)-- (1,0);
\draw [color=black] (1,0)-- (0,0);
\draw [color=black] (0.5,0.87)-- (1.5,0.87);
\draw [color=black] (1.5,0.87)-- (1,1.73);
\draw [color=black] (1,1.73)-- (0.5,0.87);
\draw [color=black] (2.5,4.33)-- (3,3.46);
\draw [color=black] (3,3.46)-- (2.5,2.6);
\draw [color=black] (2.5,2.6)-- (2,3.46);
\draw [color=black] (2,3.46)-- (2.5,4.33);
\draw [color=black] (3,3.46)-- (3.5,2.6);
\draw [color=black] (3.5,2.6)-- (3,1.73);
\draw [color=black] (3,1.73)-- (2.5,2.6);
\draw [color=black] (2.5,2.6)-- (3,3.46);
\draw [color=black] (3.5,2.6)-- (3,1.73);
\draw [color=black] (3,1.73)-- (4,1.73);
\draw [color=black] (4,1.73)-- (3.5,2.6);
\draw [color=black] (3.5,0.87)-- (4.5,0.87);
\draw [color=black] (4.5,0.87)-- (5,0);
\draw [color=black] (5,0)-- (4,0);
\draw [color=black] (4,0)-- (3.5,0.87);
\draw [color=black] (3.5,0.87)-- (3,0);
\draw [color=black] (3,0)-- (4,0);
\draw [color=black] (4,0)-- (3.5,0.87);
\draw [color=black] (3,1.73)-- (2.5,0.87);
\draw [color=black] (2.5,0.87)-- (3.5,0.87);
\draw [color=black] (3.5,0.87)-- (4,1.73);
\draw [color=black] (4,1.73)-- (3,1.73);
\draw [color=black] (2.5,0.87)-- (3.5,0.87);
\draw [color=black] (3.5,0.87)-- (3,0);
\draw [color=black] (3,0)-- (2.5,0.87);
\draw [color=black] (1.5,0.87)-- (2.5,0.87);
\draw [color=black] (2.5,0.87)-- (3,0);
\draw [color=black] (3,0)-- (2,0);
\draw [color=black] (2,0)-- (1.5,0.87);
\draw [color=black] (1,1.73)-- (2,1.73);
\draw [color=black] (2,1.73)-- (1.5,0.87);
\draw [color=black] (1.5,0.87)-- (1,1.73);
\fill [color=black] (0,0) circle (1.5pt);
\fill [color=black] (5,0) circle (1.5pt);
\fill [color=black] (2.5,4.33) circle (1.5pt);
\fill [color=black] (1,0) circle (1.5pt);
\fill [color=black] (2,0) circle (1.5pt);
\fill [color=black] (3,0) circle (1.5pt);
\fill [color=black] (4,0) circle (1.5pt);
\fill [color=black] (4.5,0.87) circle (1.5pt);
\fill [color=black] (4,1.73) circle (1.5pt);
\fill [color=black] (3.5,2.6) circle (1.5pt);
\fill [color=black] (3,3.46) circle (1.5pt);
\fill [color=black] (0.5,0.87) circle (1.5pt);
\fill [color=black] (1,1.73) circle (1.5pt);
\fill [color=black] (1.5,2.6) circle (1.5pt);
\fill [color=black] (2,3.46) circle (1.5pt);
\fill [color=black] (1.5,0.87) circle (1.5pt);
\fill [color=black] (2.5,2.6) circle (1.5pt);
\fill [color=black] (3,1.73) circle (1.5pt);
\fill [color=black] (3.5,0.87) circle (1.5pt);
\fill [color=black] (2.5,0.87) circle (1.5pt);
\fill [color=black] (2,1.73) circle (1.5pt);
\end{tiny}
\end{tikzpicture}}}

\put(145,10){${\rm wt}(P_5)=(-1)^2(1-\frac{t_2}{t_3})$}
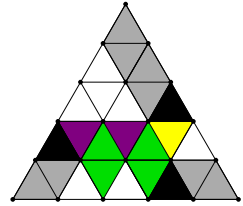
\put(120,15){\Scale[0.6]{
\begin{tikzpicture}[line cap=round,line join=round,>=triangle 45,x=1.0cm,y=1.0cm]
\clip(-2.47,-1.02) rectangle (7.63,5.36);
\begin{tiny}
\fill[color=cccccc,fill=cccccc,fill opacity=1.0] (0,0) -- (0.5,0.87) -- (1.5,0.87) -- (1,0) -- cycle;
\fill[color=black,fill=black,fill opacity=1.0] (0.5,0.87) -- (1.5,0.87) -- (1,1.73) -- cycle;
\fill[color=cccccc,fill=cccccc,fill opacity=1.0] (2.5,4.33) -- (3,3.46) -- (2.5,2.6) -- (2,3.46) -- cycle;
\fill[color=cccccc,fill=cccccc,fill opacity=1.0] (3,3.46) -- (3.5,2.6) -- (3,1.73) -- (2.5,2.6) -- cycle;
\fill[color=black,fill=black,fill opacity=1.0] (3.5,2.6) -- (3,1.73) -- (4,1.73) -- cycle;
\fill[color=cccccc,fill=cccccc,fill opacity=1.0] (3.5,0.87) -- (4.5,0.87) -- (5,0) -- (4,0) -- cycle;
\fill[color=black,fill=black,fill opacity=1.0] (3.5,0.87) -- (3,0) -- (4,0) -- cycle;
\fill[color=zzttqq,fill=zzttqq,fill opacity=1.0] (3,1.73) -- (4,1.73) -- (3.5,0.87) -- cycle;
\fill[color=ttfftt,fill=ttfftt,fill opacity=1.0] (3,1.73) -- (3.5,0.87) -- (3,0) -- (2.5,0.87) -- cycle;
\fill[color=zzqqzz,fill=zzqqzz,fill opacity=1.0] (2,1.73) -- (3,1.73) -- (2.5,0.87) -- cycle;
\fill[color=cccccc,fill=cccccc,fill opacity=1.0] (1,1.73) -- (2,1.73) -- (2.5,0.87) -- (1.5,0.87) -- cycle;
\draw [color=black] (0,0)-- (5,0);
\draw [color=black] (5,0)-- (2.5,4.33);
\draw [color=black] (2.5,4.33)-- (0,0);
\draw [color=black] (0.5,0.87)-- (1,0);
\draw [color=black] (1,1.73)-- (2,0);
\draw [color=black] (1.5,2.6)-- (3,0);
\draw [color=black] (2,3.46)-- (4,0);
\draw [color=black] (0.5,0.87)-- (4.5,0.87);
\draw [color=black] (1,1.73)-- (4,1.73);
\draw [color=black] (1.5,2.6)-- (3.5,2.6);
\draw [color=black] (2,3.46)-- (3,3.46);
\draw [color=black] (1,0)-- (3,3.46);
\draw [color=black] (3.5,2.6)-- (2,0);
\draw [color=black] (3,0)-- (4,1.73);
\draw [color=black] (4.5,0.87)-- (4,0);
\draw [color=black] (0,0)-- (0.5,0.87);
\draw [color=black] (0.5,0.87)-- (1.5,0.87);
\draw [color=black] (1.5,0.87)-- (1,0);
\draw [color=black] (1,0)-- (0,0);
\draw [color=black] (0.5,0.87)-- (1.5,0.87);
\draw [color=black] (1.5,0.87)-- (1,1.73);
\draw [color=black] (1,1.73)-- (0.5,0.87);
\draw [color=black] (2.5,4.33)-- (3,3.46);
\draw [color=black] (3,3.46)-- (2.5,2.6);
\draw [color=black] (2.5,2.6)-- (2,3.46);
\draw [color=black] (2,3.46)-- (2.5,4.33);
\draw [color=black] (3,3.46)-- (3.5,2.6);
\draw [color=black] (3.5,2.6)-- (3,1.73);
\draw [color=black] (3,1.73)-- (2.5,2.6);
\draw [color=black] (2.5,2.6)-- (3,3.46);
\draw [color=black] (3.5,2.6)-- (3,1.73);
\draw [color=black] (3,1.73)-- (4,1.73);
\draw [color=black] (4,1.73)-- (3.5,2.6);
\draw [color=black] (3.5,0.87)-- (4.5,0.87);
\draw [color=black] (4.5,0.87)-- (5,0);
\draw [color=black] (5,0)-- (4,0);
\draw [color=black] (4,0)-- (3.5,0.87);
\draw [color=black] (3.5,0.87)-- (3,0);
\draw [color=black] (3,0)-- (4,0);
\draw [color=black] (4,0)-- (3.5,0.87);
\draw [color=black] (3,1.73)-- (2.5,0.87);
\draw [color=black] (2.5,0.87)-- (3.5,0.87);
\draw [color=black] (3.5,0.87)-- (4,1.73);
\draw [color=black] (4,1.73)-- (3,1.73);
\draw [color=black] (2.5,0.87)-- (3.5,0.87);
\draw [color=black] (3.5,0.87)-- (3,0);
\draw [color=black] (3,0)-- (2.5,0.87);
\draw [color=black] (1.5,0.87)-- (2.5,0.87);
\draw [color=black] (2.5,0.87)-- (3,0);
\draw [color=black] (3,0)-- (2,0);
\draw [color=black] (2,0)-- (1.5,0.87);
\draw [color=black] (1,1.73)-- (2,1.73);
\draw [color=black] (2,1.73)-- (1.5,0.87);
\draw [color=black] (1.5,0.87)-- (1,1.73);
\fill [color=black] (0,0) circle (1.5pt);
\fill [color=black] (5,0) circle (1.5pt);
\fill [color=black] (2.5,4.33) circle (1.5pt);
\fill [color=black] (1,0) circle (1.5pt);
\fill [color=black] (2,0) circle (1.5pt);
\fill [color=black] (3,0) circle (1.5pt);
\fill [color=black] (4,0) circle (1.5pt);
\fill [color=black] (4.5,0.87) circle (1.5pt);
\fill [color=black] (4,1.73) circle (1.5pt);
\fill [color=black] (3.5,2.6) circle (1.5pt);
\fill [color=black] (3,3.46) circle (1.5pt);
\fill [color=black] (0.5,0.87) circle (1.5pt);
\fill [color=black] (1,1.73) circle (1.5pt);
\fill [color=black] (1.5,2.6) circle (1.5pt);
\fill [color=black] (2,3.46) circle (1.5pt);
\fill [color=black] (1.5,0.87) circle (1.5pt);
\fill [color=black] (2.5,2.6) circle (1.5pt);
\fill [color=black] (3,1.73) circle (1.5pt);
\fill [color=black] (3.5,0.87) circle (1.5pt);
\fill [color=black] (2.5,0.87) circle (1.5pt);
\fill [color=black] (2,1.73) circle (1.5pt);
\end{tiny}
\end{tikzpicture}}}

\put(267,10){${\rm wt}(P_6)=(-1)^3 (1-\frac{t_3}{t_4}) (1-\frac{t_2}{t_3})$}
\put(240,15){\Scale[0.6]{
\begin{tikzpicture}[line cap=round,line join=round,>=triangle 45,x=1.0cm,y=1.0cm]
\clip(-2.47,-1.02) rectangle (7.63,5.36);
\begin{tiny}
\fill[color=cccccc,fill=cccccc,fill opacity=1.0] (0,0) -- (0.5,0.87) -- (1.5,0.87) -- (1,0) -- cycle;
\fill[color=black,fill=black,fill opacity=1.0] (0.5,0.87) -- (1.5,0.87) -- (1,1.73) -- cycle;
\fill[color=cccccc,fill=cccccc,fill opacity=1.0] (2.5,4.33) -- (3,3.46) -- (2.5,2.6) -- (2,3.46) -- cycle;
\fill[color=cccccc,fill=cccccc,fill opacity=1.0] (3,3.46) -- (3.5,2.6) -- (3,1.73) -- (2.5,2.6) -- cycle;
\fill[color=black,fill=black,fill opacity=1.0] (3.5,2.6) -- (3,1.73) -- (4,1.73) -- cycle;
\fill[color=cccccc,fill=cccccc,fill opacity=1.0] (3.5,0.87) -- (4.5,0.87) -- (5,0) -- (4,0) -- cycle;
\fill[color=black,fill=black,fill opacity=1.0] (3.5,0.87) -- (3,0) -- (4,0) -- cycle;
\fill[color=zzttqq,fill=zzttqq,fill opacity=1.0] (3,1.73) -- (4,1.73) -- (3.5,0.87) -- cycle;
\fill[color=ttfftt,fill=ttfftt,fill opacity=1.0] (3,1.73) -- (3.5,0.87) -- (3,0) -- (2.5,0.87) -- cycle;
\fill[color=zzqqzz,fill=zzqqzz,fill opacity=1.0] (2,1.73) -- (3,1.73) -- (2.5,0.87) -- cycle;
\fill[color=zzqqzz,fill=zzqqzz,fill opacity=1.0] (1,1.73) -- (2,1.73) -- (1.5,0.87) -- cycle;
\fill[color=ttfftt,fill=ttfftt,fill opacity=1.0] (2,1.73) -- (1.5,0.87) -- (2,0) -- (2.5,0.87) -- cycle;
\draw [color=black] (0,0)-- (5,0);
\draw [color=black] (5,0)-- (2.5,4.33);
\draw [color=black] (2.5,4.33)-- (0,0);
\draw [color=black] (0.5,0.87)-- (1,0);
\draw [color=black] (1,1.73)-- (2,0);
\draw [color=black] (1.5,2.6)-- (3,0);
\draw [color=black] (2,3.46)-- (4,0);
\draw [color=black] (0.5,0.87)-- (4.5,0.87);
\draw [color=black] (1,1.73)-- (4,1.73);
\draw [color=black] (1.5,2.6)-- (3.5,2.6);
\draw [color=black] (2,3.46)-- (3,3.46);
\draw [color=black] (1,0)-- (3,3.46);
\draw [color=black] (3.5,2.6)-- (2,0);
\draw [color=black] (3,0)-- (4,1.73);
\draw [color=black] (4.5,0.87)-- (4,0);
\draw [color=black] (0,0)-- (0.5,0.87);
\draw [color=black] (0.5,0.87)-- (1.5,0.87);
\draw [color=black] (1.5,0.87)-- (1,0);
\draw [color=black] (1,0)-- (0,0);
\draw [color=black] (0.5,0.87)-- (1.5,0.87);
\draw [color=black] (1.5,0.87)-- (1,1.73);
\draw [color=black] (1,1.73)-- (0.5,0.87);
\draw [color=black] (2.5,4.33)-- (3,3.46);
\draw [color=black] (3,3.46)-- (2.5,2.6);
\draw [color=black] (2.5,2.6)-- (2,3.46);
\draw [color=black] (2,3.46)-- (2.5,4.33);
\draw [color=black] (3,3.46)-- (3.5,2.6);
\draw [color=black] (3.5,2.6)-- (3,1.73);
\draw [color=black] (3,1.73)-- (2.5,2.6);
\draw [color=black] (2.5,2.6)-- (3,3.46);
\draw [color=black] (3.5,2.6)-- (3,1.73);
\draw [color=black] (3,1.73)-- (4,1.73);
\draw [color=black] (4,1.73)-- (3.5,2.6);
\draw [color=black] (3.5,0.87)-- (4.5,0.87);
\draw [color=black] (4.5,0.87)-- (5,0);
\draw [color=black] (5,0)-- (4,0);
\draw [color=black] (4,0)-- (3.5,0.87);
\draw [color=black] (3.5,0.87)-- (3,0);
\draw [color=black] (3,0)-- (4,0);
\draw [color=black] (4,0)-- (3.5,0.87);
\draw [color=black] (3,1.73)-- (2.5,0.87);
\draw [color=black] (2.5,0.87)-- (3.5,0.87);
\draw [color=black] (3.5,0.87)-- (4,1.73);
\draw [color=black] (4,1.73)-- (3,1.73);
\draw [color=black] (2.5,0.87)-- (3.5,0.87);
\draw [color=black] (3.5,0.87)-- (3,0);
\draw [color=black] (3,0)-- (2.5,0.87);
\draw [color=black] (1.5,0.87)-- (2.5,0.87);
\draw [color=black] (2.5,0.87)-- (3,0);
\draw [color=black] (3,0)-- (2,0);
\draw [color=black] (2,0)-- (1.5,0.87);
\draw [color=black] (1,1.73)-- (2,1.73);
\draw [color=black] (2,1.73)-- (1.5,0.87);
\draw [color=black] (1.5,0.87)-- (1,1.73);
\fill [color=black] (0,0) circle (1.5pt);
\fill [color=black] (5,0) circle (1.5pt);
\fill [color=black] (2.5,4.33) circle (1.5pt);
\fill [color=black] (1,0) circle (1.5pt);
\fill [color=black] (2,0) circle (1.5pt);
\fill [color=black] (3,0) circle (1.5pt);
\fill [color=black] (4,0) circle (1.5pt);
\fill [color=black] (4.5,0.87) circle (1.5pt);
\fill [color=black] (4,1.73) circle (1.5pt);
\fill [color=black] (3.5,2.6) circle (1.5pt);
\fill [color=black] (3,3.46) circle (1.5pt);
\fill [color=black] (0.5,0.87) circle (1.5pt);
\fill [color=black] (1,1.73) circle (1.5pt);
\fill [color=black] (1.5,2.6) circle (1.5pt);
\fill [color=black] (2,3.46) circle (1.5pt);
\fill [color=black] (1.5,0.87) circle (1.5pt);
\fill [color=black] (2.5,2.6) circle (1.5pt);
\fill [color=black] (3,1.73) circle (1.5pt);
\fill [color=black] (3.5,0.87) circle (1.5pt);
\fill [color=black] (2.5,0.87) circle (1.5pt);
\fill [color=black] (2,1.73) circle (1.5pt);
\end{tiny}
\end{tikzpicture}}}
\end{picture}

Using double Grothendieck polynomials \cite{lascoux.schuetzenberger} (see also \cite{fulton.lascoux} and references therein), one computes $K_{01001,00101}^{10010}=-(1-\frac{t_2}{t_4})={\rm wt}(P_2)+{\rm wt}(P_3)+{\rm wt}(P_5)+{\rm wt}(P_6)$. This gives
a counterexample to Conjecture~\ref{conj:Knutson.Vakil}. Actually, this subset of four puzzles witnesses the rule of Theorem~\ref{thm:main} below.

\subsection{A modified puzzle rule} We define a {\bf modified KV-puzzle} to be a KV-puzzle
with the nonlocal condition on the second KV-piece replaced by the requirement that the second KV-piece only appears in the combination pieces
\begin{picture}(10,10)
\put(-5,-6){
\Scale[0.3]{\begin{tikzpicture}[line cap=round,line join=round,>=triangle
45,x=1.0cm,y=1.0cm]
\clip(-0.3,-3.89) rectangle (1.84,-1.8);
\fill[color=zzqqzz,fill=zzqqzz,fill opacity=1.0] (0,-2) -- (1,-2) --
(0.5,-2.87) -- cycle;
\fill[color=qqffqq,fill=qqffqq,fill opacity=1.0] (1,-2) -- (0.5,-2.87) --
(1,-3.73) -- (1.5,-2.87) -- cycle;
\end{tikzpicture}}}
\end{picture}
\; or
\begin{picture}(10,18)
\put(-5,-6){ 
\Scale[0.3]{\begin{tikzpicture}[line cap=round,line join=round,>=triangle
45,x=1.0cm,y=1.0cm]
\clip(-0.23,-3.93) rectangle (2.25,-1.86);
\fill[color=zzqqzz,fill=zzqqzz,fill opacity=1.0] (0,-2) -- (1,-2) --
(0.5,-2.87) -- cycle;
\fill[color=cccccc,fill=cccccc,fill opacity=1.0] (1,-2) -- (0.5,-2.87) --
(1.5,-2.87) -- (2,-2) -- cycle;
\fill[fill=black,fill opacity=1.0] (0.5,-2.87) -- (1.5,-2.87) -- (1,-3.73)
-- cycle;
\end{tikzpicture}}}
\end{picture} \, .

\begin{theorem}
\label{thm:main}
$K_{\lambda,\mu}^{\nu}=\sum_{P} {\rm wt}(P)$ where the sum is over all modified KV-puzzles of $\Delta_{\lambda,\mu,\nu}$.
\end{theorem}

We have a few remarks. First, the rule of Theorem~\ref{thm:main} is ``positive'' in the sense of D.~Anderson-S.~Griffeth-E.~Miller's
\cite{Anderson.Griffeth.Miller}; cf. the discussion in \cite[$\mathsection$1.4]{PeYo}.
Second, it is a natural objective 
to interpret Theorem~\ref{thm:main}
via geometric degeneration; see \cite{Coskun.Vakil, Knutson:positroid}. Third, the first author has found a tableau formulation similar to that of \cite{PeYo} to complement the puzzle rule of \cite{Knutson:positroid}
for the \emph{different} Schubert calculus problem in $K_{\sf T}({\sf X})$ of multiplying a class of a Schubert variety by that of an opposite Schubert variety; further discussion may appear elsewhere.

To prove Theorem~\ref{thm:main},  
we first give a variant of the main theorem
of \cite{PeYo}; see Section~\ref{sec:tableau_rule}. In Section~\ref{sec:proof}, we
then give a weight-preserving bijection between modified
KV-puzzles and the
objects of the rule of Section~\ref{sec:tableau_rule}.

\section{A tableau rule for $K_{\lambda,\mu}^{\nu}$}
\label{sec:tableau_rule}

We need to briefly recall the definitions of \cite[$\mathsection$~1.2--1.3]{PeYo}; there the Schubert varieties $X_{\lambda}$ are indexed by Young diagrams $\lambda$ contained
in a $k\times (n-k)$ rectangle. (Throughout, we orient Young diagrams and tableaux according to the English convention.)

An {\bf edge-labeled genomic tableau} is a filling of the boxes and horizontal edges of a skew diagram $\nu / \lambda$ with subscripted labels $i_j$, where $i$ is a positive integer and the $j$'s that appear for each $i$ form an initial interval of positive integers. Each box of $\nu / \lambda$ contains
one label, whereas the horizontal edges weakly between the southern border of $\lambda$ and the northern border of $\nu$ are filled by (possibly empty) sets of labels. A genomic edge-labeled tableau $T$ is {\bf  semistandard} if
\begin{enumerate}
\item[(S.1)] the box labels of each row strictly increase lexicographically from left to right;
\item[(S.2)] ignoring subscripts, each label is strictly less than any label strictly south in its column;
\item[(S.3)] ignoring subscripts, the labels appearing on a given edge are distinct;
\item[(S.4)] if $i_j$ appears strictly west of $i_k$, then $j \leq k$.
\end{enumerate}
Index the rows of $\nu$ from the top starting at $1$. We say a label $i_j$ is {\bf too high} if it appears weakly above the north edge of row $i$.
We refer to the collection of all $i_j$'s (for fixed $i,j$) as a {\bf gene}. 
The {\bf content} of $T$ is the composition $(\alpha_1, \alpha_2, \dots )$ where $\alpha_i$ is greatest so that $i_{\alpha_i}$ is a gene of $T$. 

Recall that in the classical tableau theory, a semistandard tableau $S$ is \emph{ballot} if, reading the labels down columns from right to left, we obtain a word $W$ with the following property: For each $i$, every initial segment of $W$ contains at least as many $i$'s as $(i+1)$'s. Given an edge-labeled genomic tableau $T$, choose one label from each gene and delete all others; now delete all subscripts. We say $T$ is {\bf ballot} if, regardless of our choices from genes, the resulting tableau (possibly containing holes) is necessarily ballot in the above classical sense. (In the case of multiple labels on a edge, read them from least to greatest.)

We now diverge slightly from the treatment of \cite{PeYo}, borrowing notation from \cite{Thomas.Yong:H_T}. Given a box $\x$ in an edge-labeled genomic tableau $T$, we say $\x$ is {\bf starrable} if it contains $i_j$, is in row $>i$, and $i_{j+1}$ is not a box label to its immediate right. 
Let ${\tt StarBallotGen}_\mu(\nu/\lambda)$ be the set of all ballot semistandard edge-labeled genomic tableaux of shape $\nu / \lambda$ and content $\mu$ with no label too high, where the label of each starrable box may freely be marked by $\star$ or not. The tableau $T$ illustrated in Figure~\ref{fig:thetab} is an element of ${\tt StarBallotGen}_{(10,5,3)}((15,8,5)/(12,2,1))$. There are three starrable boxes in $T$, in only one of which the label has been starred.

Let ${\sf Man}(\x)$ denote the length of any $\{\uparrow,\rightarrow\}$-lattice path from the southwest corner of $k \times (n-k)$ to the northwest corner of $\x$. For $\x$ in row $r$ containing $i_j^\star$, set ${\tt starfactor}(\x) := 1 - \frac{t_{{\sf Man}(\x) + 1}}{t_{r - i + \mu_i - j + 1 + {\sf Man}(\x)}}$. For an edge label $\ell = i_j$ in the southern edge of $\x$ in row $r$, set ${\tt edgefactor} := 1 - \frac{t_{{\sf Man}(\x)}}{t_{r - i + \mu_i - j + 1 + {\sf Man}(\x)}}$. 
Finally for $T\in {\tt StarBallotGen}_\mu(\nu/\lambda)$, define 
\[{\widehat {\rm wt}}(T):=(-1)^{{\hat d}(T)}\times \prod_\ell {\tt edgefactor}(\ell) \times
\prod_{\sf x} {\tt starfactor}(\x) ;\]
here the products are respectively over edge labels $\ell$ and boxes $\x$ containing starred labels, while
${\hat d}(T):=\text{$\#$(labels in $T$)} + \text{$\#$($\star$'s in $T$)} - |\mu|$.
Let 
\[{\hat L}_{\lambda,\mu}^{\nu}:=\sum_{T} {\widehat {\rm wt}}(T),\] 
where the sum is over
all $T\in {\tt StarBallotGen}_\mu(\nu/\lambda)$.

We need a reformulation of \cite[Theorem~1.3]{PeYo}; the proof is a simple application
of the ``inclusion-exclusion'' identity $\prod_{i\in [m]}a_i =\sum_{S\subseteq [m]}(-1)^{|S|}\prod_{i\in S}(1-a_i)$.

\begin{theorem}
\label{thm:reformulated}
$K_{\lambda,\mu}^{\nu}={\hat L}_{\lambda,\mu}^{\nu}.$ \qed
\end{theorem}

\begin{example}
\label{exa:fourtableaux}
Let $k=2, n=5$ and $\lambda=(2,0)$, $\mu=(1,0)$ and $\nu=(3,1)$. The four tableaux contributing
to ${\hat L}_{\lambda,\mu}^{\nu}$ are 
\[\ \ \ \ 
\begin{picture}(120,63)
\ytableausetup{boxsize=1.7em}
\put(0,43){$\ytableaushort{ {*(lightgray)\blank} {*(lightgray)\blank} {1_1}, 
{1_1}}$}
\put(-10,0){${\widehat {\rm wt}}(T_2)=-1$}
\end{picture}
\begin{picture}(120,63)
\ytableausetup{boxsize=1.7em}
\put(0,43){$\ytableaushort{ {*(lightgray)\blank} {*(lightgray)\blank} {1_1}, 
{1_1}}$}
\put(26,39){$1_1$}
\put(-45,0){${\widehat {\rm wt}}(T_3)\!=\!(-1)^2(1-\frac{t_3}{t_4})$}
\end{picture}
\begin{picture}(120,63)
\ytableausetup{boxsize=1.7em}
\put(0,43){$\ytableaushort{ {*(lightgray)\blank} {*(lightgray)\blank} {1_1}, 
{1_1^\star}}$}
\put(-38,0){${\widehat {\rm wt}}(T_5)\!=\!(-1)^2(1-\frac{t_2}{t_3})$}
\end{picture}
\begin{picture}(120,63)
\ytableausetup{boxsize=1.7em}
\put(0,43){$\ytableaushort{ {*(lightgray)\blank} {*(lightgray)\blank} {1_1}, 
{1_1^\star}}$}
\put(26,39){$1_1$}
\put(-30,0){${\widehat {\rm wt}}(T_6)\!=\!(-1)^3(1\!-\!\frac{t_3}{t_4})(1\!-\!\frac{t_2}{t_3})$}
\end{picture}
\]
Our indexing of these tableaux alludes to the precise connection to the four puzzles $P_2,P_3,P_5$ and $P_6$ of Section~\ref{sec:puzzle_conjecture}, as explained in the next section.\qed
\end{example}

\section{Proof of Theorem~\ref{thm:main}: bijecting the tableau and puzzle rules}
\label{sec:proof}

\subsection{Description of the bijection}
To relate the modifed KV-puzzle rule of Theorem~\ref{thm:main} with the tableau rule of
Theorem~\ref{thm:reformulated}, we give a variant of T.~Tao's ``proof without words''
\cite{Vakil:annals} (and its modification by K.~Purbhoo \cite{purbhoo}) that bijects cohomological puzzles 
(using the first three pieces) and a tableau Littlewood-Richardson
rule. An extension of this proof for equivariant puzzles (i.e., fillings that additionally use the equivariant piece)
was given by V.~Kreiman \cite{Kreiman}; we also encorporate
elements of his bijection in our analysis. 

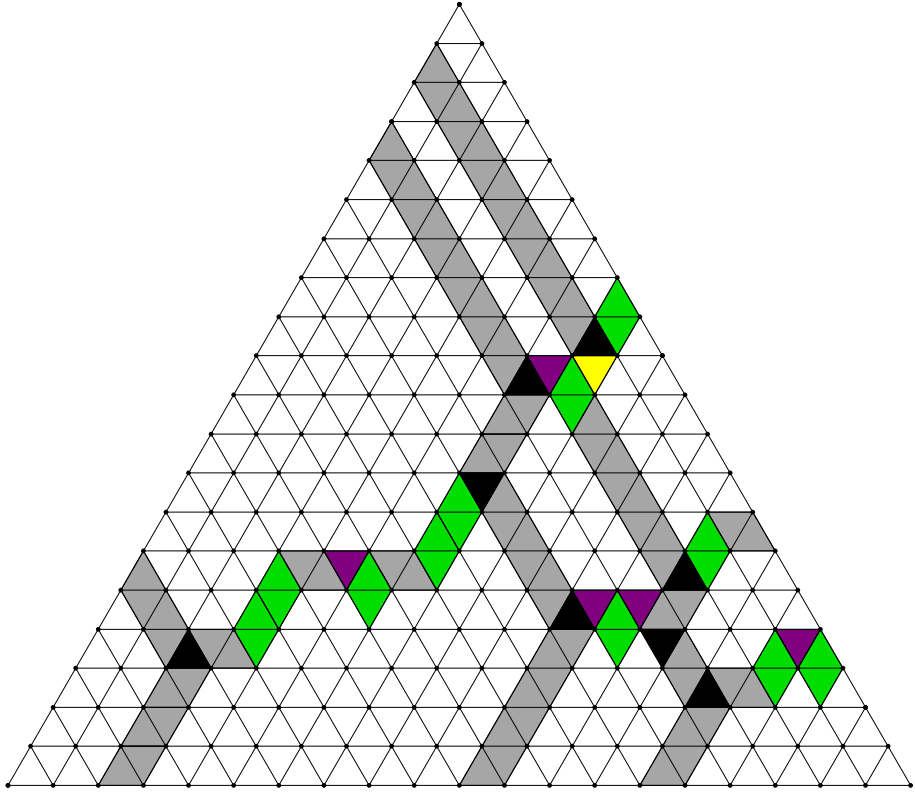
\begin{figure}[t]
\Scale[0.6]{
\begin{tikzpicture}[line cap=round,line join=round,>=triangle 45,x=1.0cm,y=1.0cm]
\clip(-1.58,-0.79) rectangle (20.9,17.84);
\fill[color=black,fill=eqeqeq,fill opacity=1.0] (2,0) -- (3,0) -- (3.5,0.87) -- (2.5,0.87) -- cycle;
\fill[color=black,fill=eqeqeq,fill opacity=1.0] (2.5,0.87) -- (3.5,0.87) -- (4,1.73) -- (3,1.73) -- cycle;
\fill[color=black,fill=eqeqeq,fill opacity=1.0] (3,1.73) -- (4,1.73) -- (4.5,2.6) -- (3.5,2.6) -- cycle;
\fill[fill=black,fill opacity=1.0] (3.5,2.6) -- (4.5,2.6) -- (4,3.46) -- cycle;
\fill[color=black,fill=eqeqeq,fill opacity=1.0] (4,3.46) -- (5,3.46) -- (5.5,2.6) -- (4.5,2.6) -- cycle;
\fill[color=black,fill=qqffqq,fill opacity=1.0] (5,3.46) -- (5.5,4.33) -- (6,3.46) -- (5.5,2.6) -- cycle;
\fill[color=black,fill=qqffqq,fill opacity=1.0] (5.5,4.33) -- (6,5.2) -- (6.5,4.33) -- (6,3.46) -- cycle;
\fill[color=black,fill=eqeqeq,fill opacity=1.0] (6,5.2) -- (7,5.2) -- (7.5,4.33) -- (6.5,4.33) -- cycle;
\fill[color=black,fill=ccwwff,fill opacity=1.0] (7,5.2) -- (8,5.2) -- (7.5,4.33) -- cycle;
\fill[color=black,fill=qqffqq,fill opacity=1.0] (8,5.2) -- (7.5,4.33) -- (8,3.46) -- (8.5,4.33) -- cycle;
\fill[color=black,fill=eqeqeq,fill opacity=1.0] (8,5.2) -- (9,5.2) -- (9.5,4.33) -- (8.5,4.33) -- cycle;
\fill[color=black,fill=qqffqq,fill opacity=1.0] (9.5,4.33) -- (10,5.2) -- (9.5,6.06) -- (9,5.2) -- cycle;
\fill[color=black,fill=qqffqq,fill opacity=1.0] (9.5,6.06) -- (10,6.93) -- (10.5,6.06) -- (10,5.2) -- cycle;
\fill[fill=black,fill opacity=1.0] (10,6.93) -- (11,6.93) -- (10.5,6.06) -- cycle;
\fill[color=black,fill=eqeqeq,fill opacity=1.0] (10,6.93) -- (11,6.93) -- (11.5,7.79) -- (10.5,7.79) -- cycle;
\fill[color=black,fill=eqeqeq,fill opacity=1.0] (10.5,7.79) -- (11.5,7.79) -- (12,8.66) -- (11,8.66) -- cycle;
\fill[fill=black,fill opacity=1.0] (11,8.66) -- (12,8.66) -- (11.5,9.53) -- cycle;
\fill[color=black,fill=ccwwff,fill opacity=1.0] (11.5,9.53) -- (12,8.66) -- (12.5,9.53) -- cycle;
\fill[fill=black,fill opacity=1.0] (12.5,9.53) -- (13.5,9.53) -- (13,10.39) -- cycle;
\fill[color=black,fill=qqffqq,fill opacity=1.0] (13.5,9.53) -- (14,10.39) -- (13.5,11.26) -- (13,10.39) -- cycle;
\fill[color=black,fill=eqeqeq,fill opacity=1.0] (3.5,2.6) -- (4,3.46) -- (3,5.2) -- (2.5,4.33) -- cycle;
\fill[color=black,fill=eqeqeq,fill opacity=1.0] (11,8.66) -- (11.5,9.53) -- (8.5,14.72) -- (8,13.86) -- cycle;
\fill[color=black,fill=eqeqeq,fill opacity=1.0] (12.5,9.53) -- (13,10.39) -- (9.5,16.45) -- (9,15.59) -- cycle;
\fill[color=black,fill=eqeqeq,fill opacity=1.0] (10.5,6.06) -- (11,6.93) -- (12.5,4.33) -- (12,3.46) -- cycle;
\fill[fill=black,fill opacity=1.0] (12,3.46) -- (13,3.46) -- (12.5,4.33) -- cycle;
\fill[color=black,fill=eqeqeq,fill opacity=1.0] (12,3.46) -- (10,0) -- (11,0) -- (13,3.46) -- cycle;
\fill[color=black,fill=ccwwff,fill opacity=1.0] (12.5,4.33) -- (13,3.46) -- (13.5,4.33) -- cycle;
\fill[color=black,fill=qqffqq,fill opacity=1.0] (13.5,4.33) -- (13,3.46) -- (13.5,2.6) -- (14,3.46) -- cycle;
\fill[color=black,fill=ccwwff,fill opacity=1.0] (13.5,4.33) -- (14.5,4.33) -- (14,3.46) -- cycle;
\fill[color=black,fill=black,fill opacity=1.0] (14,3.46) -- (14.5,2.6) -- (15,3.46) -- cycle;
\fill[color=black,fill=eqeqeq,fill opacity=1.0] (14.5,4.33) -- (15.5,4.33) -- (15,3.46) -- (14,3.46) -- cycle;
\fill[fill=black,fill opacity=1.0] (14.5,4.33) -- (15.5,4.33) -- (15,5.2) -- cycle;
\fill[color=black,fill=eqeqeq,fill opacity=1.0] (13,8.66) -- (12.5,7.79) -- (14.5,4.33) -- (15,5.2) -- cycle;
\fill[color=black,fill=qqffqq,fill opacity=1.0] (15.5,4.33) -- (15,5.2) -- (15.5,6.06) -- (16,5.2) -- cycle;
\fill[color=black,fill=eqeqeq,fill opacity=1.0] (15.5,6.06) -- (16.5,6.06) -- (17,5.2) -- (16,5.2) -- cycle;
\fill[fill=black,fill opacity=1.0] (15,1.73) -- (15.5,2.6) -- (16,1.73) -- cycle;
\fill[color=black,fill=eqeqeq,fill opacity=1.0] (14,0) -- (15,1.73) -- (16,1.73) -- (15,0) -- cycle;
\fill[color=black,fill=eqeqeq,fill opacity=1.0] (14.5,2.6) -- (15,3.46) -- (15.5,2.6) -- (15,1.73) -- cycle;
\fill[color=black,fill=eqeqeq,fill opacity=1.0] (15.5,2.6) -- (16.5,2.6) -- (17,1.73) -- (16,1.73) -- cycle;
\fill[color=black,fill=qqffqq,fill opacity=1.0] (16.5,2.6) -- (17,1.73) -- (17.5,2.6) -- (17,3.46) -- cycle;
\fill[color=black,fill=qqffqq,fill opacity=1.0] (18,3.46) -- (17.5,2.6) -- (18,1.73) -- (18.5,2.6) -- cycle;
\fill[color=black,fill=ccwwff,fill opacity=1.0] (17,3.46) -- (18,3.46) -- (17.5,2.6) -- cycle;
\fill[color=black,fill=qqffqq,fill opacity=1.0] (12.5,9.53) -- (12,8.66) -- (12.5,7.79) -- (13,8.66) -- cycle;
\fill[color=black,fill=zzttqq,fill opacity=1.0] (12.5,9.53) -- (13.5,9.53) -- (13,8.66) -- cycle;
\draw [color=black] (0,0)-- (20,0);
\draw [color=black] (20,0)-- (10,17.32);
\draw [color=black] (10,17.32)-- (0,0);
\draw [color=black] (9.5,16.45)-- (10.5,16.45);
\draw [color=black] (9,15.59)-- (11,15.59);
\draw [color=black] (8.5,14.72)-- (11.5,14.72);
\draw [color=black] (8,13.86)-- (12,13.86);
\draw [color=black] (7.5,12.99)-- (12.5,12.99);
\draw [color=black] (7,12.12)-- (13,12.12);
\draw [color=black] (6.5,11.26)-- (13.5,11.26);
\draw [color=black] (6,10.39)-- (14,10.39);
\draw [color=black] (5.5,9.53)-- (14.5,9.53);
\draw [color=black] (5,8.66)-- (15,8.66);
\draw [color=black] (4.5,7.79)-- (15.5,7.79);
\draw [color=black] (4,6.93)-- (16,6.93);
\draw [color=black] (3.5,6.06)-- (16.5,6.06);
\draw [color=black] (3,5.2)-- (17,5.2);
\draw [color=black] (2.5,4.33)-- (17.5,4.33);
\draw [color=black] (2,3.46)-- (18,3.46);
\draw [color=black] (1.5,2.6)-- (18.5,2.6);
\draw [color=black] (1,1.73)-- (19,1.73);
\draw [color=black] (19.5,0.87)-- (0.5,0.87);
\draw [color=black] (0.5,0.87)-- (1,0);
\draw [color=black] (2,0)-- (1,1.73);
\draw [color=black] (1.5,2.6)-- (3,0);
\draw [color=black] (4,0)-- (2,3.46);
\draw [color=black] (2.5,4.33)-- (5,0);
\draw [color=black] (6,0)-- (3,5.2);
\draw [color=black] (3.5,6.06)-- (7,0);
\draw [color=black] (8,0)-- (4,6.93);
\draw [color=black] (4.5,7.79)-- (9,0);
\draw [color=black] (10,0)-- (5,8.66);
\draw [color=black] (5.5,9.53)-- (11,0);
\draw [color=black] (12,0)-- (6,10.39);
\draw [color=black] (6.5,11.26)-- (13,0);
\draw [color=black] (14,0)-- (7,12.12);
\draw [color=black] (7.5,12.99)-- (15,0);
\draw [color=black] (16,0)-- (8,13.86);
\draw [color=black] (8.5,14.72)-- (17,0);
\draw [color=black] (18,0)-- (9,15.59);
\draw [color=black] (9.5,16.45)-- (19,0);
\draw [color=black] (19.5,0.87)-- (19,0);
\draw [color=black] (18,0)-- (19,1.73);
\draw [color=black] (18.5,2.6)-- (17,0);
\draw [color=black] (16,0)-- (18,3.46);
\draw [color=black] (17.5,4.33)-- (15,0);
\draw [color=black] (14,0)-- (17,5.2);
\draw [color=black] (16.5,6.06)-- (13,0);
\draw [color=black] (12,0)-- (16,6.93);
\draw [color=black] (15.5,7.79)-- (11,0);
\draw [color=black] (10,0)-- (15,8.66);
\draw [color=black] (14.5,9.53)-- (9,0);
\draw [color=black] (8,0)-- (14,10.39);
\draw [color=black] (13.5,11.26)-- (7,0);
\draw [color=black] (6,0)-- (13,12.12);
\draw [color=black] (12.5,12.99)-- (5,0);
\draw [color=black] (4,0)-- (12,13.86);
\draw [color=black] (11.5,14.72)-- (3,0);
\draw [color=black] (2,0)-- (11,15.59);
\draw [color=black] (10.5,16.45)-- (1,0);
\draw [color=black] (2,0)-- (3,0);
\draw [color=black] (3,0)-- (3.5,0.87);
\draw [color=black] (3.5,0.87)-- (2.5,0.87);
\draw [color=black] (2.5,0.87)-- (2,0);
\draw [color=black] (2.5,0.87)-- (3.5,0.87);
\draw [color=black] (3.5,0.87)-- (4,1.73);
\draw [color=black] (4,1.73)-- (3,1.73);
\draw [color=black] (3,1.73)-- (2.5,0.87);
\draw [color=black] (3,1.73)-- (4,1.73);
\draw [color=black] (4,1.73)-- (4.5,2.6);
\draw [color=black] (4.5,2.6)-- (3.5,2.6);
\draw [color=black] (3.5,2.6)-- (3,1.73);
\draw [color=black] (3.5,2.6)-- (4.5,2.6);
\draw [color=black] (4.5,2.6)-- (4,3.46);
\draw [color=black] (4,3.46)-- (3.5,2.6);
\draw [color=black] (4,3.46)-- (5,3.46);
\draw [color=black] (5,3.46)-- (5.5,2.6);
\draw [color=black] (5.5,2.6)-- (4.5,2.6);
\draw [color=black] (4.5,2.6)-- (4,3.46);
\draw [color=black] (5,3.46)-- (5.5,4.33);
\draw [color=black] (5.5,4.33)-- (6,3.46);
\draw [color=black] (6,3.46)-- (5.5,2.6);
\draw [color=black] (5.5,2.6)-- (5,3.46);
\draw [color=black] (5.5,4.33)-- (6,5.2);
\draw [color=black] (6,5.2)-- (6.5,4.33);
\draw [color=black] (6.5,4.33)-- (6,3.46);
\draw [color=black] (6,3.46)-- (5.5,4.33);
\draw [color=black] (6,5.2)-- (7,5.2);
\draw [color=black] (7,5.2)-- (7.5,4.33);
\draw [color=black] (7.5,4.33)-- (6.5,4.33);
\draw [color=black] (6.5,4.33)-- (6,5.2);
\draw [color=black] (7,5.2)-- (8,5.2);
\draw [color=black] (8,5.2)-- (7.5,4.33);
\draw [color=black] (7.5,4.33)-- (7,5.2);
\draw [color=black] (8,5.2)-- (7.5,4.33);
\draw [color=black] (7.5,4.33)-- (8,3.46);
\draw [color=black] (8,3.46)-- (8.5,4.33);
\draw [color=black] (8.5,4.33)-- (8,5.2);
\draw [color=black] (8,5.2)-- (9,5.2);
\draw [color=black] (9,5.2)-- (9.5,4.33);
\draw [color=black] (9.5,4.33)-- (8.5,4.33);
\draw [color=black] (8.5,4.33)-- (8,5.2);
\draw [color=black] (9.5,4.33)-- (10,5.2);
\draw [color=black] (10,5.2)-- (9.5,6.06);
\draw [color=black] (9.5,6.06)-- (9,5.2);
\draw [color=black] (9,5.2)-- (9.5,4.33);
\draw [color=black] (9.5,6.06)-- (10,6.93);
\draw [color=black] (10,6.93)-- (10.5,6.06);
\draw [color=black] (10.5,6.06)-- (10,5.2);
\draw [color=black] (10,5.2)-- (9.5,6.06);
\draw [color=black] (10,6.93)-- (11,6.93);
\draw [color=black] (11,6.93)-- (10.5,6.06);
\draw [color=black] (10.5,6.06)-- (10,6.93);
\draw [color=black] (10,6.93)-- (11,6.93);
\draw [color=black] (11,6.93)-- (11.5,7.79);
\draw [color=black] (11.5,7.79)-- (10.5,7.79);
\draw [color=black] (10.5,7.79)-- (10,6.93);
\draw [color=black] (10.5,7.79)-- (11.5,7.79);
\draw [color=black] (11.5,7.79)-- (12,8.66);
\draw [color=black] (12,8.66)-- (11,8.66);
\draw [color=black] (11,8.66)-- (10.5,7.79);
\draw [color=black] (11,8.66)-- (12,8.66);
\draw [color=black] (12,8.66)-- (11.5,9.53);
\draw [color=black] (11.5,9.53)-- (11,8.66);
\draw [color=black] (11.5,9.53)-- (12,8.66);
\draw [color=black] (12,8.66)-- (12.5,9.53);
\draw [color=black] (12.5,9.53)-- (11.5,9.53);
\draw [color=black] (12.5,9.53)-- (13.5,9.53);
\draw [color=black] (13.5,9.53)-- (13,10.39);
\draw [color=black] (13,10.39)-- (12.5,9.53);
\draw [color=black] (13.5,9.53)-- (14,10.39);
\draw [color=black] (14,10.39)-- (13.5,11.26);
\draw [color=black] (13.5,11.26)-- (13,10.39);
\draw [color=black] (13,10.39)-- (13.5,9.53);
\draw [color=black] (3.5,2.6)-- (4,3.46);
\draw [color=black] (4,3.46)-- (3,5.2);
\draw [color=black] (3,5.2)-- (2.5,4.33);
\draw [color=black] (2.5,4.33)-- (3.5,2.6);
\draw [color=black] (11,8.66)-- (11.5,9.53);
\draw [color=black] (11.5,9.53)-- (8.5,14.72);
\draw [color=black] (8.5,14.72)-- (8,13.86);
\draw [color=black] (8,13.86)-- (11,8.66);
\draw [color=black] (12.5,9.53)-- (13,10.39);
\draw [color=black] (13,10.39)-- (9.5,16.45);
\draw [color=black] (9.5,16.45)-- (9,15.59);
\draw [color=black] (9,15.59)-- (12.5,9.53);
\draw [color=black] (10.5,6.06)-- (11,6.93);
\draw [color=black] (11,6.93)-- (12.5,4.33);
\draw [color=black] (12.5,4.33)-- (12,3.46);
\draw [color=black] (12,3.46)-- (10.5,6.06);
\draw [color=black] (12,3.46)-- (13,3.46);
\draw [color=black] (13,3.46)-- (12.5,4.33);
\draw [color=black] (12.5,4.33)-- (12,3.46);
\draw [color=black] (12,3.46)-- (10,0);
\draw [color=black] (10,0)-- (11,0);
\draw [color=black] (11,0)-- (13,3.46);
\draw [color=black] (13,3.46)-- (12,3.46);
\draw [color=black] (12.5,4.33)-- (13,3.46);
\draw [color=black] (13,3.46)-- (13.5,4.33);
\draw [color=black] (13.5,4.33)-- (12.5,4.33);
\draw [color=black] (13.5,4.33)-- (13,3.46);
\draw [color=black] (13,3.46)-- (13.5,2.6);
\draw [color=black] (13.5,2.6)-- (14,3.46);
\draw [color=black] (14,3.46)-- (13.5,4.33);
\draw [color=black] (13.5,4.33)-- (14.5,4.33);
\draw [color=black] (14.5,4.33)-- (14,3.46);
\draw [color=black] (14,3.46)-- (13.5,4.33);
\draw [color=black] (14,3.46)-- (14.5,2.6);
\draw [color=black] (14.5,2.6)-- (15,3.46);
\draw [color=black] (15,3.46)-- (14,3.46);
\draw [color=black] (14.5,4.33)-- (15.5,4.33);
\draw [color=black] (15.5,4.33)-- (15,3.46);
\draw [color=black] (15,3.46)-- (14,3.46);
\draw [color=black] (14,3.46)-- (14.5,4.33);
\draw [color=black] (14.5,4.33)-- (15.5,4.33);
\draw [color=black] (15.5,4.33)-- (15,5.2);
\draw [color=black] (15,5.2)-- (14.5,4.33);
\draw [color=black] (13,8.66)-- (12.5,7.79);
\draw [color=black] (12.5,7.79)-- (14.5,4.33);
\draw [color=black] (14.5,4.33)-- (15,5.2);
\draw [color=black] (15,5.2)-- (13,8.66);
\draw [color=black] (15.5,4.33)-- (15,5.2);
\draw [color=black] (15,5.2)-- (15.5,6.06);
\draw [color=black] (15.5,6.06)-- (16,5.2);
\draw [color=black] (16,5.2)-- (15.5,4.33);
\draw [color=black] (15.5,6.06)-- (16.5,6.06);
\draw [color=black] (16.5,6.06)-- (17,5.2);
\draw [color=black] (17,5.2)-- (16,5.2);
\draw [color=black] (16,5.2)-- (15.5,6.06);
\draw [color=black] (15,1.73)-- (15.5,2.6);
\draw [color=black] (15.5,2.6)-- (16,1.73);
\draw [color=black] (16,1.73)-- (15,1.73);
\draw [color=black] (14,0)-- (15,1.73);
\draw [color=black] (15,1.73)-- (16,1.73);
\draw [color=black] (16,1.73)-- (15,0);
\draw [color=black] (15,0)-- (14,0);
\draw [color=black] (14.5,2.6)-- (15,3.46);
\draw [color=black] (15,3.46)-- (15.5,2.6);
\draw [color=black] (15.5,2.6)-- (15,1.73);
\draw [color=black] (15,1.73)-- (14.5,2.6);
\draw [color=black] (15.5,2.6)-- (16.5,2.6);
\draw [color=black] (16.5,2.6)-- (17,1.73);
\draw [color=black] (17,1.73)-- (16,1.73);
\draw [color=black] (16,1.73)-- (15.5,2.6);
\draw [color=black] (16.5,2.6)-- (17,1.73);
\draw [color=black] (17,1.73)-- (17.5,2.6);
\draw [color=black] (17.5,2.6)-- (17,3.46);
\draw [color=black] (17,3.46)-- (16.5,2.6);
\draw [color=black] (18,3.46)-- (17.5,2.6);
\draw [color=black] (17.5,2.6)-- (18,1.73);
\draw [color=black] (18,1.73)-- (18.5,2.6);
\draw [color=black] (18.5,2.6)-- (18,3.46);
\draw [color=black] (17,3.46)-- (18,3.46);
\draw [color=black] (18,3.46)-- (17.5,2.6);
\draw [color=black] (17.5,2.6)-- (17,3.46);
\draw [color=black] (12.5,9.53)-- (12,8.66);
\draw [color=black] (12,8.66)-- (12.5,7.79);
\draw [color=black] (12.5,7.79)-- (13,8.66);
\draw [color=black] (13,8.66)-- (12.5,9.53);
\draw [color=black] (12.5,9.53)-- (13.5,9.53);
\draw [color=black] (13.5,9.53)-- (13,8.66);
\draw [color=black] (13,8.66)-- (12.5,9.53);
\begin{scriptsize}
\fill [color=black] (0,0) circle (1.5pt);
\fill [color=black] (20,0) circle (1.5pt);
\fill [color=black] (1,0) circle (1.5pt);
\fill [color=black] (2,0) circle (1.5pt);
\fill [color=black] (3,0) circle (1.5pt);
\fill [color=black] (4,0) circle (1.5pt);
\fill [color=black] (5,0) circle (1.5pt);
\fill [color=black] (6,0) circle (1.5pt);
\fill [color=black] (7,0) circle (1.5pt);
\fill [color=black] (8,0) circle (1.5pt);
\fill [color=black] (9,0) circle (1.5pt);
\fill [color=black] (10,0) circle (1.5pt);
\fill [color=black] (11,0) circle (1.5pt);
\fill [color=black] (12,0) circle (1.5pt);
\fill [color=black] (13,0) circle (1.5pt);
\fill [color=black] (14,0) circle (1.5pt);
\fill [color=black] (15,0) circle (1.5pt);
\fill [color=black] (16,0) circle (1.5pt);
\fill [color=black] (17,0) circle (1.5pt);
\fill [color=black] (18,0) circle (1.5pt);
\fill [color=black] (19,0) circle (1.5pt);
\fill [color=black] (10,17.32) circle (1.5pt);
\fill [color=black] (0.5,0.87) circle (1.5pt);
\fill [color=black] (0,0) circle (1.5pt);
\fill [color=black] (10,17.32) circle (1.5pt);
\fill [color=black] (1,1.73) circle (1.5pt);
\fill [color=black] (1.5,2.6) circle (1.5pt);
\fill [color=black] (2,3.46) circle (1.5pt);
\fill [color=black] (2.5,4.33) circle (1.5pt);
\fill [color=black] (3,5.2) circle (1.5pt);
\fill [color=black] (3.5,6.06) circle (1.5pt);
\fill [color=black] (4,6.93) circle (1.5pt);
\fill [color=black] (4.5,7.79) circle (1.5pt);
\fill [color=black] (5,8.66) circle (1.5pt);
\fill [color=black] (5.5,9.53) circle (1.5pt);
\fill [color=black] (6,10.39) circle (1.5pt);
\fill [color=black] (6.5,11.26) circle (1.5pt);
\fill [color=black] (7,12.12) circle (1.5pt);
\fill [color=black] (7.5,12.99) circle (1.5pt);
\fill [color=black] (8,13.86) circle (1.5pt);
\fill [color=black] (8.5,14.72) circle (1.5pt);
\fill [color=black] (8.5,14.72) circle (1.5pt);
\fill [color=black] (9.5,16.45) circle (1.5pt);
\fill [color=black] (10,17.32) circle (1.5pt);
\fill [color=black] (9,15.59) circle (1.5pt);
\fill [color=black] (19.5,0.87) circle (1.5pt);
\fill [color=black] (19,1.73) circle (1.5pt);
\fill [color=black] (10.5,16.45) circle (1.5pt);
\fill [color=black] (11,15.59) circle (1.5pt);
\fill [color=black] (11.5,14.72) circle (1.5pt);
\fill [color=black] (12,13.86) circle (1.5pt);
\fill [color=black] (12.5,12.99) circle (1.5pt);
\fill [color=black] (13,12.12) circle (1.5pt);
\fill [color=black] (13.5,11.26) circle (1.5pt);
\fill [color=black] (14,10.39) circle (1.5pt);
\fill [color=black] (14.5,9.53) circle (1.5pt);
\fill [color=black] (14.5,9.53) circle (1.5pt);
\fill [color=black] (15,8.66) circle (1.5pt);
\fill [color=black] (15.5,7.79) circle (1.5pt);
\fill [color=black] (16,6.93) circle (1.5pt);
\fill [color=black] (16.5,6.06) circle (1.5pt);
\fill [color=black] (17,5.2) circle (1.5pt);
\fill [color=black] (17.5,4.33) circle (1.5pt);
\fill [color=black] (18,3.46) circle (1.5pt);
\fill [color=black] (19,1.73) circle (1.5pt);
\fill [color=black] (18.5,2.6) circle (1.5pt);
\fill [color=black] (3.5,0.87) circle (1.5pt);
\fill [color=black] (2.5,0.87) circle (1.5pt);
\fill [color=black] (4,1.73) circle (1.5pt);
\fill [color=black] (3,1.73) circle (1.5pt);
\fill [color=black] (4.5,2.6) circle (1.5pt);
\fill [color=black] (3.5,2.6) circle (1.5pt);
\fill [color=black] (4,3.46) circle (1.5pt);
\fill [color=black] (5,3.46) circle (1.5pt);
\fill [color=black] (5.5,2.6) circle (1.5pt);
\fill [color=black] (5.5,4.33) circle (1.5pt);
\fill [color=black] (6,3.46) circle (1.5pt);
\fill [color=black] (6,5.2) circle (1.5pt);
\fill [color=black] (6.5,4.33) circle (1.5pt);
\fill [color=black] (7,5.2) circle (1.5pt);
\fill [color=black] (7.5,4.33) circle (1.5pt);
\fill [color=black] (8,5.2) circle (1.5pt);
\fill [color=black] (8,3.46) circle (1.5pt);
\fill [color=black] (8.5,4.33) circle (1.5pt);
\fill [color=black] (9,5.2) circle (1.5pt);
\fill [color=black] (9.5,4.33) circle (1.5pt);
\fill [color=black] (10,5.2) circle (1.5pt);
\fill [color=black] (9.5,6.06) circle (1.5pt);
\fill [color=black] (10,6.93) circle (1.5pt);
\fill [color=black] (10.5,6.06) circle (1.5pt);
\fill [color=black] (11,6.93) circle (1.5pt);
\fill [color=black] (11.5,7.79) circle (1.5pt);
\fill [color=black] (10.5,7.79) circle (1.5pt);
\fill [color=black] (12,8.66) circle (1.5pt);
\fill [color=black] (11,8.66) circle (1.5pt);
\fill [color=black] (11.5,9.53) circle (1.5pt);
\fill [color=black] (12.5,9.53) circle (1.5pt);
\fill [color=black] (13.5,9.53) circle (1.5pt);
\fill [color=black] (13,8.66) circle (1.5pt);
\fill [color=black] (12.5,7.79) circle (1.5pt);
\fill [color=black] (13,10.39) circle (1.5pt);
\fill [color=black] (12.5,4.33) circle (1.5pt);
\fill [color=black] (12,3.46) circle (1.5pt);
\fill [color=black] (13,3.46) circle (1.5pt);
\fill [color=black] (12.5,2.6) circle (1.5pt);
\fill [color=black] (11.5,2.6) circle (1.5pt);
\fill [color=black] (12,1.73) circle (1.5pt);
\fill [color=black] (10.5,2.6) circle (1.5pt);
\fill [color=black] (11,1.73) circle (1.5pt);
\fill [color=black] (10,1.73) circle (1.5pt);
\fill [color=black] (10.5,0.87) circle (1.5pt);
\fill [color=black] (9.5,0.87) circle (1.5pt);
\fill [color=black] (9,1.73) circle (1.5pt);
\fill [color=black] (8,1.73) circle (1.5pt);
\fill [color=black] (8.5,0.87) circle (1.5pt);
\fill [color=black] (6,1.73) circle (1.5pt);
\fill [color=black] (6.5,0.87) circle (1.5pt);
\fill [color=black] (5.5,0.87) circle (1.5pt);
\fill [color=black] (13.5,4.33) circle (1.5pt);
\fill [color=black] (13.5,2.6) circle (1.5pt);
\fill [color=black] (14,3.46) circle (1.5pt);
\fill [color=black] (14.5,4.33) circle (1.5pt);
\fill [color=black] (14.5,2.6) circle (1.5pt);
\fill [color=black] (15,3.46) circle (1.5pt);
\fill [color=black] (15.5,4.33) circle (1.5pt);
\fill [color=black] (15,5.2) circle (1.5pt);
\fill [color=black] (15.5,6.06) circle (1.5pt);
\fill [color=black] (16,5.2) circle (1.5pt);
\fill [color=black] (15,1.73) circle (1.5pt);
\fill [color=black] (15.5,2.6) circle (1.5pt);
\fill [color=black] (16,1.73) circle (1.5pt);
\fill [color=black] (16.5,2.6) circle (1.5pt);
\fill [color=black] (17,1.73) circle (1.5pt);
\fill [color=black] (17.5,2.6) circle (1.5pt);
\fill [color=black] (17,3.46) circle (1.5pt);
\fill [color=black] (18,1.73) circle (1.5pt);
\fill [color=black] (1.5,0.87) circle (1.5pt);
\fill [color=black] (2,1.73) circle (1.5pt);
\fill [color=black] (2.5,2.6) circle (1.5pt);
\fill [color=black] (3,3.46) circle (1.5pt);
\fill [color=black] (3.5,4.33) circle (1.5pt);
\fill [color=black] (4,5.2) circle (1.5pt);
\fill [color=black] (4.5,6.06) circle (1.5pt);
\fill [color=black] (5,6.93) circle (1.5pt);
\fill [color=black] (5.5,7.79) circle (1.5pt);
\fill [color=black] (6,8.66) circle (1.5pt);
\fill [color=black] (6.5,9.53) circle (1.5pt);
\fill [color=black] (7,10.39) circle (1.5pt);
\fill [color=black] (7.5,11.26) circle (1.5pt);
\fill [color=black] (8,12.12) circle (1.5pt);
\fill [color=black] (8.5,12.99) circle (1.5pt);
\fill [color=black] (9,13.86) circle (1.5pt);
\fill [color=black] (9.5,14.72) circle (1.5pt);
\fill [color=black] (10,15.59) circle (1.5pt);
\fill [color=black] (10.5,14.72) circle (1.5pt);
\fill [color=black] (10,13.86) circle (1.5pt);
\fill [color=black] (9.5,12.99) circle (1.5pt);
\fill [color=black] (10.5,12.99) circle (1.5pt);
\fill [color=black] (11,13.86) circle (1.5pt);
\fill [color=black] (11.5,12.99) circle (1.5pt);
\fill [color=black] (12,12.12) circle (1.5pt);
\fill [color=black] (11,12.12) circle (1.5pt);
\fill [color=black] (10,12.12) circle (1.5pt);
\fill [color=black] (9,12.12) circle (1.5pt);
\fill [color=black] (8.5,11.26) circle (1.5pt);
\fill [color=black] (9.5,11.26) circle (1.5pt);
\fill [color=black] (10.5,11.26) circle (1.5pt);
\fill [color=black] (11.5,11.26) circle (1.5pt);
\fill [color=black] (12.5,11.26) circle (1.5pt);
\fill [color=black] (12,10.39) circle (1.5pt);
\fill [color=black] (11,10.39) circle (1.5pt);
\fill [color=black] (10,10.39) circle (1.5pt);
\fill [color=black] (9,10.39) circle (1.5pt);
\fill [color=black] (8,10.39) circle (1.5pt);
\fill [color=black] (7.5,9.53) circle (1.5pt);
\fill [color=black] (8.5,9.53) circle (1.5pt);
\fill [color=black] (9.5,9.53) circle (1.5pt);
\fill [color=black] (10.5,9.53) circle (1.5pt);
\fill [color=black] (14,8.66) circle (1.5pt);
\fill [color=black] (10,8.66) circle (1.5pt);
\fill [color=black] (9,8.66) circle (1.5pt);
\fill [color=black] (8,8.66) circle (1.5pt);
\fill [color=black] (7,8.66) circle (1.5pt);
\fill [color=black] (6.5,7.79) circle (1.5pt);
\fill [color=black] (7.5,7.79) circle (1.5pt);
\fill [color=black] (8.5,7.79) circle (1.5pt);
\fill [color=black] (9.5,7.79) circle (1.5pt);
\fill [color=black] (9,6.93) circle (1.5pt);
\fill [color=black] (8,6.93) circle (1.5pt);
\fill [color=black] (7,6.93) circle (1.5pt);
\fill [color=black] (6,6.93) circle (1.5pt);
\fill [color=black] (5.5,6.06) circle (1.5pt);
\fill [color=black] (6.5,6.06) circle (1.5pt);
\fill [color=black] (7.5,6.06) circle (1.5pt);
\fill [color=black] (8.5,6.06) circle (1.5pt);
\fill [color=black] (5,5.2) circle (1.5pt);
\fill [color=black] (4.5,4.33) circle (1.5pt);
\fill [color=black] (13.5,7.79) circle (1.5pt);
\fill [color=black] (14.5,7.79) circle (1.5pt);
\fill [color=black] (15,6.93) circle (1.5pt);
\fill [color=black] (14,6.93) circle (1.5pt);
\fill [color=black] (13,6.93) circle (1.5pt);
\fill [color=black] (12,6.93) circle (1.5pt);
\fill [color=black] (11.5,6.06) circle (1.5pt);
\fill [color=black] (12.5,6.06) circle (1.5pt);
\fill [color=black] (13.5,6.06) circle (1.5pt);
\fill [color=black] (14.5,6.06) circle (1.5pt);
\fill [color=black] (14,5.2) circle (1.5pt);
\fill [color=black] (13,5.2) circle (1.5pt);
\fill [color=black] (12,5.2) circle (1.5pt);
\fill [color=black] (11,5.2) circle (1.5pt);
\fill [color=black] (11.5,4.33) circle (1.5pt);
\fill [color=black] (10.5,4.33) circle (1.5pt);
\fill [color=black] (9,3.46) circle (1.5pt);
\fill [color=black] (7,3.46) circle (1.5pt);
\fill [color=black] (5,1.73) circle (1.5pt);
\fill [color=black] (4.5,0.87) circle (1.5pt);
\fill [color=black] (6.5,2.6) circle (1.5pt);
\fill [color=black] (7.5,2.6) circle (1.5pt);
\fill [color=black] (7,1.73) circle (1.5pt);
\fill [color=black] (7.5,0.87) circle (1.5pt);
\fill [color=black] (8.5,2.6) circle (1.5pt);
\fill [color=black] (9.5,2.6) circle (1.5pt);
\fill [color=black] (10,3.46) circle (1.5pt);
\fill [color=black] (11,3.46) circle (1.5pt);
\fill [color=black] (11.5,0.87) circle (1.5pt);
\fill [color=black] (12.5,0.87) circle (1.5pt);
\fill [color=black] (13,1.73) circle (1.5pt);
\fill [color=black] (14,1.73) circle (1.5pt);
\fill [color=black] (13.5,0.87) circle (1.5pt);
\fill [color=black] (14.5,0.87) circle (1.5pt);
\fill [color=black] (15.5,0.87) circle (1.5pt);
\fill [color=black] (16.5,0.87) circle (1.5pt);
\fill [color=black] (17.5,0.87) circle (1.5pt);
\fill [color=black] (18.5,0.87) circle (1.5pt);
\fill [color=black] (16,3.46) circle (1.5pt);
\fill [color=black] (16.5,4.33) circle (1.5pt);
\end{scriptsize}
\end{tikzpicture}}
\caption{A ``generic'' modified KV-puzzle $P$ ($k=3,n=20$).}
\label{fig:bigex}
\end{figure}

Figure~\ref{fig:bigex} gives a ``generic'' example of a (modified) KV-puzzle $P$. We will define a {\bf track} 
$\pi_i$ from the $i$th
$1$ (from the left) on the $\nu$-boundary of $\Delta_{\lambda,\mu,\nu}$ to the $i$th $1$ (from the top) on the $\mu$-boundary. 
To do this, we describe the {\bf flow} through the (oriented, non-KV) puzzle pieces that use a $1$ and four {\bf combination pieces} (possible ways one can use
the KV-pieces under the rules for a modified KV-puzzle): 
\begin{itemize}
\item[(A.1)]
\begin{picture}(10,10)
\put(-5,-1){
\Scale[0.3]{\begin{tikzpicture}[line cap=round,line join=round,>=triangle
45,x=1.0cm,y=1.0cm]
\clip(-0.21,-0.18) rectangle (1.72,1.04);
\fill[color=cccccc,fill=cccccc,fill opacity=1.0] (0,0) -- (0.5,0.87) --
(1.5,0.87) -- (1,0) -- cycle;
\end{tikzpicture}}}
\end{picture}
\ : go northeast
\item[(A.2)]  %cappiece
\begin{picture}(10,10)
\put(-5,-2){
\Scale[0.3]{
\begin{tikzpicture}[line cap=round,line join=round,>=triangle
45,x=1.0cm,y=1.0cm]
\clip(0.67,-0.37) rectangle (2.27,1.04);
\fill[color=black,fill=black,fill opacity=1.0] (1.5,0.87) -- (1,0) --
(2,0) -- cycle;
\end{tikzpicture}}}
\end{picture}
\ : go north then northeast
\item[(A.3)] 
\begin{picture}(10,10)
\put(-5,-2){
\Scale[0.3]{
\begin{tikzpicture}[line cap=round,line join=round,>=triangle
45,x=1.0cm,y=1.0cm]
\clip(1.25,-0.23) rectangle (3.29,1.04);
\fill[color=cccccc,fill=cccccc,fill opacity=1.0] (1.5,0.87) -- (2.5,0.87) --
(3,0) -- (2,0) -- cycle;
\end{tikzpicture}}}
\end{picture}
\ : go left to right
\item[(A.4)] %greenpiece
\begin{picture}(10,10)
\put(-5,-6){
\Scale[0.3]{\begin{tikzpicture}[line cap=round,line join=round,>=triangle 45,x=1.0cm,y=1.0cm]
\clip(1.12,-0.4) rectangle (2.76,2.1);
\fill[color=qqffqq,fill=qqffqq,fill opacity=1.0] (2,1.73) -- (1.5,0.87) -- (2,0) -- (2.5,0.87) -- cycle;
\draw [color=qqffqq] (2,1.73)-- (1.5,0.87);
\draw [color=qqffqq] (1.5,0.87)-- (2,0);
\draw [color=qqffqq] (2,0)-- (2.5,0.87);
\draw [color=qqffqq] (2.5,0.87)-- (2,1.73);
\end{tikzpicture}}}
\end{picture}
\ : go northeast
\item[(A.5)] 
\begin{picture}(10,10)
\put(-5,-6){
\Scale[0.3]{\begin{tikzpicture}[line cap=round,line join=round,>=triangle
45,x=1.0cm,y=1.0cm]
\clip(-0.3,-3.89) rectangle (1.84,-1.8);
\fill[color=zzqqzz,fill=zzqqzz,fill opacity=1.0] (0,-2) -- (1,-2) --
(0.5,-2.87) -- cycle;
\fill[color=qqffqq,fill=qqffqq,fill opacity=1.0] (1,-2) -- (0.5,-2.87) --
(1,-3.73) -- (1.5,-2.87) -- cycle;
\end{tikzpicture}}}
\end{picture}
\ : go in through the north $\backslash$ of the purple triangle, come out northeast from the purple gash into the southwest $\backslash$ of the 
green rhombus and pass northeast through this rhombus
\item[(A.6)] 
\begin{picture}(10,10)
\put(-5,-2){
\Scale[0.3]{\begin{tikzpicture}[line cap=round,line join=round,>=triangle
45,x=1.0cm,y=1.0cm]
\clip(0.72,0.58) rectangle (2.19,1.92);
\fill[color=black,fill=black,fill opacity=1.0] (1,1.73) -- (2,1.73) --
(1.5,0.87) -- cycle;
\end{tikzpicture}}}
\end{picture}
\ : come in through the left side and out the top
\item[(A.7)] 
\begin{picture}(10,10)
\put(-5,-6){
\Scale[0.3]{
\begin{tikzpicture}[line cap=round,line join=round,>=triangle
45,x=1.0cm,y=1.0cm]
\clip(0.2,-3.98) rectangle (2.24,-1.63);
\fill[color=qqffqq,fill=qqffqq,fill opacity=1.0] (1,-2) -- (0.5,-2.87) --
(1,-3.73) -- (1.5,-2.87) -- cycle;
\fill[color=zzttqq,fill=zzttqq,fill opacity=1.0] (1,-2) -- (1.5,-2.87) --
(2,-2) -- cycle;
\draw [color=qqffqq] (1,-2)-- (0.5,-2.87);
\draw [color=qqffqq] (0.5,-2.87)-- (1,-3.73);
\draw [color=qqffqq] (1,-3.73)-- (1.5,-2.87);
\draw [color=qqffqq] (1.5,-2.87)-- (1,-2);
\draw [color=zzttqq] (1,-2)-- (1.5,-2.87);
\draw [color=zzttqq] (1.5,-2.87)-- (2,-2);
\draw [color=zzttqq] (2,-2)-- (1,-2);
\end{tikzpicture}}}
\end{picture}
\ : come in through the southwest side of the green rhombus and out the top of the yellow triangle
\item[(A.8)]
\begin{picture}(10,12)
\put(-5,-6){ 
\Scale[0.3]{\begin{tikzpicture}[line cap=round,line join=round,>=triangle
45,x=1.0cm,y=1.0cm]
\clip(-0.23,-3.93) rectangle (2.25,-1.86);
\fill[color=zzqqzz,fill=zzqqzz,fill opacity=1.0] (0,-2) -- (1,-2) --
(0.5,-2.87) -- cycle;
\fill[color=cccccc,fill=cccccc,fill opacity=1.0] (1,-2) -- (0.5,-2.87) --
(1.5,-2.87) -- (2,-2) -- cycle;
\fill[fill=black,fill opacity=1.0] (0.5,-2.87) -- (1.5,-2.87) -- (1,-3.73)
-- cycle;
\end{tikzpicture}}}
\end{picture}
\ : come in through the north $\backslash$ of the purple triangle, out the gash into the 
$\backslash$ of the
\begin{picture}(10,10)
\put(-5,-2){
\Scale[0.3]{\begin{tikzpicture}[line cap=round,line join=round,>=triangle
45,x=1.0cm,y=1.0cm]
\clip(0.72,0.58) rectangle (2.19,1.92);
\fill[color=black,fill=black,fill opacity=1.0] (1,1.73) -- (2,1.73) --
(1.5,0.87) -- cycle;
\end{tikzpicture}}}
\end{picture}, out the --- of 
\begin{picture}(10,10)
\put(-5,-2){
\Scale[0.3]{\begin{tikzpicture}[line cap=round,line join=round,>=triangle
45,x=1.0cm,y=1.0cm]
\clip(0.72,0.58) rectangle (2.19,1.92);
\fill[color=black,fill=black,fill opacity=1.0] (1,1.73) -- (2,1.73) --
(1.5,0.87) -- cycle;
\end{tikzpicture}}}
\end{picture}
into the bottom of the grey rhombus and out its top
\item[(A.9)] 
\begin{picture}(10,12)
\put(-5,-6){
\Scale[0.3]{
\begin{tikzpicture}[line cap=round,line join=round,>=triangle
45,x=1.0cm,y=1.0cm]
\clip(-0.28,-3.92) rectangle (2.26,-1.9);
\fill[color=zzqqzz,fill=zzqqzz,fill opacity=1.0] (0,-2) -- (1,-2) --
(0.5,-2.87) -- cycle;
\fill[color=qqffqq,fill=qqffqq,fill opacity=1.0] (1,-2) -- (0.5,-2.87) --
(1,-3.73) -- (1.5,-2.87) -- cycle;
\fill[color=zzttqq,fill=zzttqq,fill opacity=1.0] (1,-2) -- (1.5,-2.87) --
(2,-2) -- cycle;
\end{tikzpicture}}}
\end{picture}
\ : come into the north $\backslash$ of the purple triangle, out the gash into the southwest $\backslash$ of the green rhombus
and out the northeast $\backslash$ into the left side of the yellow triangle and then go out the 
--- of that triangle.
\end{itemize}

%It will be convenient to define a {\bf cap} to be an apex up black triangle and a {\bf cup} to be an apex down %black triangle. Similarly, we define {\bf rightgrey} to be a grey rhombus that points right and {\bf leftgrey} to %be a grey rhombus that points left.
  
Thinking of the (combination) pieces in (A.1)--(A.9) as letters of an alphabet, we can encode the northmost track in $P$ (from Figure~\ref{fig:bigex}) as the word 
%$\rightgrey$
\[
\begin{picture}(10,10)
\put(-5,-1){
\Scale[0.3]{\begin{tikzpicture}[line cap=round,line join=round,>=triangle
45,x=1.0cm,y=1.0cm]
\clip(-0.21,-0.18) rectangle (1.72,1.04);
\fill[color=cccccc,fill=cccccc,fill opacity=1.0] (0,0) -- (0.5,0.87) --
(1.5,0.87) -- (1,0) -- cycle;
\end{tikzpicture}}}
\end{picture}^{\ 3} \ 
\begin{picture}(10,10)
\put(-5,-2){
\Scale[0.3]{
\begin{tikzpicture}[line cap=round,line join=round,>=triangle
45,x=1.0cm,y=1.0cm]
\clip(0.67,-0.37) rectangle (2.27,1.04);
\fill[color=black,fill=black,fill opacity=1.0] (1.5,0.87) -- (1,0) --
(2,0) -- cycle;
\end{tikzpicture}}}
\end{picture}
 \ 
\begin{picture}(10,10)
\put(-5,-2){
\Scale[0.3]{
\begin{tikzpicture}[line cap=round,line join=round,>=triangle
45,x=1.0cm,y=1.0cm]
\clip(1.25,-0.23) rectangle (3.29,1.04);
\fill[color=cccccc,fill=cccccc,fill opacity=1.0] (1.5,0.87) -- (2.5,0.87) --
(3,0) -- (2,0) -- cycle;
\end{tikzpicture}}}
\end{picture}
 \ \begin{picture}(10,10)
\put(-5,-6){
\Scale[0.3]{\begin{tikzpicture}[line cap=round,line join=round,>=triangle 45,x=1.0cm,y=1.0cm]
\clip(1.12,-0.4) rectangle (2.76,2.1);
\fill[color=qqffqq,fill=qqffqq,fill opacity=1.0] (2,1.73) -- (1.5,0.87) -- (2,0) -- (2.5,0.87) -- cycle;
\draw [color=qqffqq] (2,1.73)-- (1.5,0.87);
\draw [color=qqffqq] (1.5,0.87)-- (2,0);
\draw [color=qqffqq] (2,0)-- (2.5,0.87);
\draw [color=qqffqq] (2.5,0.87)-- (2,1.73);
\end{tikzpicture}}}
\end{picture}^{\ 2} \ 
\begin{picture}(10,10)
\put(-5,-2){
\Scale[0.3]{
\begin{tikzpicture}[line cap=round,line join=round,>=triangle
45,x=1.0cm,y=1.0cm]
\clip(1.25,-0.23) rectangle (3.29,1.04);
\fill[color=cccccc,fill=cccccc,fill opacity=1.0] (1.5,0.87) -- (2.5,0.87) --
(3,0) -- (2,0) -- cycle;
\end{tikzpicture}}}
\end{picture}
\ 
\begin{picture}(10,10)
\put(-5,-6){
\Scale[0.3]{\begin{tikzpicture}[line cap=round,line join=round,>=triangle
45,x=1.0cm,y=1.0cm]
\clip(-0.3,-3.89) rectangle (1.84,-1.8);
\fill[color=zzqqzz,fill=zzqqzz,fill opacity=1.0] (0,-2) -- (1,-2) --
(0.5,-2.87) -- cycle;
\fill[color=qqffqq,fill=qqffqq,fill opacity=1.0] (1,-2) -- (0.5,-2.87) --
(1,-3.73) -- (1.5,-2.87) -- cycle;
\end{tikzpicture}}}
\end{picture}
 \ 
\begin{picture}(10,10)
\put(-5,-2){
\Scale[0.3]{
\begin{tikzpicture}[line cap=round,line join=round,>=triangle
45,x=1.0cm,y=1.0cm]
\clip(1.25,-0.23) rectangle (3.29,1.04);
\fill[color=cccccc,fill=cccccc,fill opacity=1.0] (1.5,0.87) -- (2.5,0.87) --
(3,0) -- (2,0) -- cycle;
\end{tikzpicture}}}
\end{picture}
\ 
\begin{picture}(10,10)
\put(-5,-6){
\Scale[0.3]{\begin{tikzpicture}[line cap=round,line join=round,>=triangle 45,x=1.0cm,y=1.0cm]
\clip(1.12,-0.4) rectangle (2.76,2.1);
\fill[color=qqffqq,fill=qqffqq,fill opacity=1.0] (2,1.73) -- (1.5,0.87) -- (2,0) -- (2.5,0.87) -- cycle;
\draw [color=qqffqq] (2,1.73)-- (1.5,0.87);
\draw [color=qqffqq] (1.5,0.87)-- (2,0);
\draw [color=qqffqq] (2,0)-- (2.5,0.87);
\draw [color=qqffqq] (2.5,0.87)-- (2,1.73);
\end{tikzpicture}}}
\end{picture}^{\ 2} \ 
\begin{picture}(10,10)
\put(-5,-2){
\Scale[0.3]{\begin{tikzpicture}[line cap=round,line join=round,>=triangle
45,x=1.0cm,y=1.0cm]
\clip(0.72,0.58) rectangle (2.19,1.92);
\fill[color=black,fill=black,fill opacity=1.0] (1,1.73) -- (2,1.73) --
(1.5,0.87) -- cycle;
\end{tikzpicture}}}
\end{picture}
 \ \begin{picture}(10,10)
\put(-5,-1){
\Scale[0.3]{\begin{tikzpicture}[line cap=round,line join=round,>=triangle
45,x=1.0cm,y=1.0cm]
\clip(-0.21,-0.18) rectangle (1.72,1.04);
\fill[color=cccccc,fill=cccccc,fill opacity=1.0] (0,0) -- (0.5,0.87) --
(1.5,0.87) -- (1,0) -- cycle;
\end{tikzpicture}}}
\end{picture}^{\ 2}  
\ \begin{picture}(10,10)
\put(-5,-2){
\Scale[0.3]{
\begin{tikzpicture}[line cap=round,line join=round,>=triangle
45,x=1.0cm,y=1.0cm]
\clip(0.67,-0.37) rectangle (2.27,1.04);
\fill[color=black,fill=black,fill opacity=1.0] (1.5,0.87) -- (1,0) --
(2,0) -- cycle;
\end{tikzpicture}}}
\end{picture} \ 
\begin{picture}(10,12)
\put(-5,-6){
\Scale[0.3]{
\begin{tikzpicture}[line cap=round,line join=round,>=triangle
45,x=1.0cm,y=1.0cm]
\clip(-0.28,-3.92) rectangle (2.26,-1.9);
\fill[color=zzqqzz,fill=zzqqzz,fill opacity=1.0] (0,-2) -- (1,-2) --
(0.5,-2.87) -- cycle;
\fill[color=qqffqq,fill=qqffqq,fill opacity=1.0] (1,-2) -- (0.5,-2.87) --
(1,-3.73) -- (1.5,-2.87) -- cycle;
\fill[color=zzttqq,fill=zzttqq,fill opacity=1.0] (1,-2) -- (1.5,-2.87) --
(2,-2) -- cycle;
\end{tikzpicture}}}
\end{picture}
 \ \ \begin{picture}(10,10)
\put(-5,-2){
\Scale[0.3]{
\begin{tikzpicture}[line cap=round,line join=round,>=triangle
45,x=1.0cm,y=1.0cm]
\clip(0.67,-0.37) rectangle (2.27,1.04);
\fill[color=black,fill=black,fill opacity=1.0] (1.5,0.87) -- (1,0) --
(2,0) -- cycle;
\end{tikzpicture}}}
\end{picture}
\ 
\begin{picture}(10,10)
\put(-5,-6){
\Scale[0.3]{\begin{tikzpicture}[line cap=round,line join=round,>=triangle 45,x=1.0cm,y=1.0cm]
\clip(1.12,-0.4) rectangle (2.76,2.1);
\fill[color=qqffqq,fill=qqffqq,fill opacity=1.0] (2,1.73) -- (1.5,0.87) -- (2,0) -- (2.5,0.87) -- cycle;
\draw [color=qqffqq] (2,1.73)-- (1.5,0.87);
\draw [color=qqffqq] (1.5,0.87)-- (2,0);
\draw [color=qqffqq] (2,0)-- (2.5,0.87);
\draw [color=qqffqq] (2.5,0.87)-- (2,1.73);
\end{tikzpicture}}}
\end{picture}.\]

Recall, if $\kappa$ is a letter/word in some alphabet, then
the {\bf Kleene star} is
$\kappa^*:=\{\emptyset, \kappa, \kappa\kappa,\ldots\}$.

\begin{proposition}[Decomposition of $\pi_i$]
\label{claim:canonical}
The list of (combination) pieces that appear in $\pi_i$, as read from southwest to northeast, 
is a word from the following formal grammar: 
\begin{equation}
\label{eqn:canonicalform}
\text{\tt boxes}[\text{\tt edges} \ \text{\tt startrow} \ \text{\tt boxes}]^* \ \text{\tt edges}
\end{equation}
where 
\begin{align*}
\text{\tt boxes}:= & \begin{picture}(10,10)
\put(-5,-1){
\Scale[0.3]{\begin{tikzpicture}[line cap=round,line join=round,>=triangle
45,x=1.0cm,y=1.0cm]
\clip(-0.21,-0.18) rectangle (1.72,1.04);
\fill[color=cccccc,fill=cccccc,fill opacity=1.0] (0,0) -- (0.5,0.87) --
(1.5,0.87) -- (1,0) -- cycle;
\end{tikzpicture}}}
\end{picture}^{\ *} \ 
\begin{picture}(10,10)
\put(-5,-2){
\Scale[0.3]{
\begin{tikzpicture}[line cap=round,line join=round,>=triangle
45,x=1.0cm,y=1.0cm]
\clip(0.67,-0.37) rectangle (2.27,1.04);
\fill[color=black,fill=black,fill opacity=1.0] (1.5,0.87) -- (1,0) --
(2,0) -- cycle;
\end{tikzpicture}}}
\end{picture}
\\
\text{\tt edges}:= & [\begin{picture}(10,10)
\put(-5,-2){
\Scale[0.3]{
\begin{tikzpicture}[line cap=round,line join=round,>=triangle
45,x=1.0cm,y=1.0cm]
\clip(1.25,-0.23) rectangle (3.29,1.04);
\fill[color=cccccc,fill=cccccc,fill opacity=1.0] (1.5,0.87) -- (2.5,0.87) --
(3,0) -- (2,0) -- cycle;
\end{tikzpicture}}}
\end{picture}^{\ *} \begin{picture}(10,10)
\put(-5,-6){
\Scale[0.3]{\begin{tikzpicture}[line cap=round,line join=round,>=triangle 45,x=1.0cm,y=1.0cm]
\clip(1.12,-0.4) rectangle (2.76,2.1);
\fill[color=qqffqq,fill=qqffqq,fill opacity=1.0] (2,1.73) -- (1.5,0.87) -- (2,0) -- (2.5,0.87) -- cycle;
\draw [color=qqffqq] (2,1.73)-- (1.5,0.87);
\draw [color=qqffqq] (1.5,0.87)-- (2,0);
\draw [color=qqffqq] (2,0)-- (2.5,0.87);
\draw [color=qqffqq] (2.5,0.87)-- (2,1.73);
\end{tikzpicture}}}
\end{picture}^* 
\begin{picture}(10,10)
\put(-5,-2){
\Scale[0.3]{
\begin{tikzpicture}[line cap=round,line join=round,>=triangle
45,x=1.0cm,y=1.0cm]
\clip(1.25,-0.23) rectangle (3.29,1.04);
\fill[color=cccccc,fill=cccccc,fill opacity=1.0] (1.5,0.87) -- (2.5,0.87) --
(3,0) -- (2,0) -- cycle;
\end{tikzpicture}}}
\end{picture}
^{\ *} \
\begin{picture}(10,10)
\put(-5,-6){
\Scale[0.3]{\begin{tikzpicture}[line cap=round,line join=round,>=triangle
45,x=1.0cm,y=1.0cm]
\clip(-0.3,-3.89) rectangle (1.84,-1.8);
\fill[color=zzqqzz,fill=zzqqzz,fill opacity=1.0] (0,-2) -- (1,-2) --
(0.5,-2.87) -- cycle;
\fill[color=qqffqq,fill=qqffqq,fill opacity=1.0] (1,-2) -- (0.5,-2.87) --
(1,-3.73) -- (1.5,-2.87) -- cycle;
\end{tikzpicture}}}
\end{picture}^{\ *}]^*\\
\text{\tt startrow}:= & 
\begin{picture}(10,10)
\put(-5,-2){
\Scale[0.3]{\begin{tikzpicture}[line cap=round,line join=round,>=triangle
45,x=1.0cm,y=1.0cm]
\clip(0.72,0.58) rectangle (2.19,1.92);
\fill[color=black,fill=black,fill opacity=1.0] (1,1.73) -- (2,1.73) --
(1.5,0.87) -- cycle;
\end{tikzpicture}}}
\end{picture}
 \cup 
\begin{picture}(10,10)
\put(-5,-6){
\Scale[0.3]{
\begin{tikzpicture}[line cap=round,line join=round,>=triangle
45,x=1.0cm,y=1.0cm]
\clip(0.2,-3.98) rectangle (2.24,-1.63);
\fill[color=qqffqq,fill=qqffqq,fill opacity=1.0] (1,-2) -- (0.5,-2.87) --
(1,-3.73) -- (1.5,-2.87) -- cycle;
\fill[color=zzttqq,fill=zzttqq,fill opacity=1.0] (1,-2) -- (1.5,-2.87) --
(2,-2) -- cycle;
\draw [color=qqffqq] (1,-2)-- (0.5,-2.87);
\draw [color=qqffqq] (0.5,-2.87)-- (1,-3.73);
\draw [color=qqffqq] (1,-3.73)-- (1.5,-2.87);
\draw [color=qqffqq] (1.5,-2.87)-- (1,-2);
\draw [color=zzttqq] (1,-2)-- (1.5,-2.87);
\draw [color=zzttqq] (1.5,-2.87)-- (2,-2);
\draw [color=zzttqq] (2,-2)-- (1,-2);
\end{tikzpicture}}}
\end{picture}
\ \cup
\begin{picture}(10,12)
\put(-5,-6){ 
\Scale[0.3]{\begin{tikzpicture}[line cap=round,line join=round,>=triangle
45,x=1.0cm,y=1.0cm]
\clip(-0.23,-3.93) rectangle (2.25,-1.86);
\fill[color=zzqqzz,fill=zzqqzz,fill opacity=1.0] (0,-2) -- (1,-2) --
(0.5,-2.87) -- cycle;
\fill[color=cccccc,fill=cccccc,fill opacity=1.0] (1,-2) -- (0.5,-2.87) --
(1.5,-2.87) -- (2,-2) -- cycle;
\fill[fill=black,fill opacity=1.0] (0.5,-2.87) -- (1.5,-2.87) -- (1,-3.73)
-- cycle;
\end{tikzpicture}}}
\end{picture}
\ \cup \begin{picture}(10,12)
\put(-5,-6){
\Scale[0.3]{
\begin{tikzpicture}[line cap=round,line join=round,>=triangle
45,x=1.0cm,y=1.0cm]
\clip(-0.28,-3.92) rectangle (2.26,-1.9);
\fill[color=zzqqzz,fill=zzqqzz,fill opacity=1.0] (0,-2) -- (1,-2) --
(0.5,-2.87) -- cycle;
\fill[color=qqffqq,fill=qqffqq,fill opacity=1.0] (1,-2) -- (0.5,-2.87) --
(1,-3.73) -- (1.5,-2.87) -- cycle;
\fill[color=zzttqq,fill=zzttqq,fill opacity=1.0] (1,-2) -- (1.5,-2.87) --
(2,-2) -- cycle;
\end{tikzpicture}}}
\end{picture}
\end{align*}
\end{proposition}
\begin{proof}
By inspection of the rules for modified KV-puzzles.
\end{proof}

The remaining filling of the puzzle is forced, which we explain in two steps. First there is the {\bf NWray} of each 
\begin{picture}(10,10)
\put(-5,-2){
\Scale[0.3]{
\begin{tikzpicture}[line cap=round,line join=round,>=triangle
45,x=1.0cm,y=1.0cm]
\clip(0.67,-0.37) rectangle (2.27,1.04);
\fill[color=black,fill=black,fill opacity=1.0] (1.5,0.87) -- (1,0) --
(2,0) -- cycle;
\end{tikzpicture}}}
\end{picture}, i.e., the (possibly empty) path of upward pointing grey rhombi 
\begin{picture}(10,10)
\put(-5,-6){
\Scale[0.3]{\begin{tikzpicture}[line cap=round,line join=round,>=triangle 45,x=1.0cm,y=1.0cm]
\clip(1.12,-0.4) rectangle (2.76,2.1);
\fill[color=cccccc,fill=cccccc,fill opacity=1.0] (2,1.73) -- (1.5,0.87) -- (2,0) -- (2.5,0.87) -- cycle;
\draw [color=cccccc] (2,1.73)-- (1.5,0.87);
\draw [color=cccccc] (1.5,0.87)-- (2,0);
\draw [color=cccccc] (2,0)-- (2.5,0.87);
\draw [color=cccccc] (2.5,0.87)-- (2,1.73);
\end{tikzpicture}}}
\end{picture}
growing 
from the $/$ of this 
\begin{picture}(10,10)
\put(-5,-2){
\Scale[0.3]{
\begin{tikzpicture}[line cap=round,line join=round,>=triangle
45,x=1.0cm,y=1.0cm]
\clip(0.67,-0.37) rectangle (2.27,1.04);
\fill[color=black,fill=black,fill opacity=1.0] (1.5,0.87) -- (1,0) --
(2,0) -- cycle;
\end{tikzpicture}}}
\end{picture}.

\begin{lemma}
\label{lemma:NWray}
The NWray of
\begin{picture}(10,10)
\put(-5,-2){
\Scale[0.3]{
\begin{tikzpicture}[line cap=round,line join=round,>=triangle
45,x=1.0cm,y=1.0cm]
\clip(0.67,-0.37) rectangle (2.27,1.04);
\fill[color=black,fill=black,fill opacity=1.0] (1.5,0.87) -- (1,0) --
(2,0) -- cycle;
\end{tikzpicture}}}
\end{picture}
ends either at the $\lambda$-boundary of $\Delta$ or with a piece from {\tt startrow}. In the latter
case, the shared edge is the south-then-eastmost edge of the
(combination) piece.
\end{lemma}
\begin{proof}
The north $/$ of \begin{picture}(10,10)
\put(-5,-6){
\Scale[0.3]{\begin{tikzpicture}[line cap=round,line join=round,>=triangle 45,x=1.0cm,y=1.0cm]
\clip(1.12,-0.4) rectangle (2.76,2.1);
\fill[color=cccccc,fill=cccccc,fill opacity=1.0] (2,1.73) -- (1.5,0.87) -- (2,0) -- (2.5,0.87) -- cycle;
\draw [color=cccccc] (2,1.73)-- (1.5,0.87);
\draw [color=cccccc] (1.5,0.87)-- (2,0);
\draw [color=cccccc] (2,0)-- (2.5,0.87);
\draw [color=cccccc] (2.5,0.87)-- (2,1.73);
\end{tikzpicture}}}
\end{picture} is labeled $1$. 
By inspection, the only (combination) pieces that can
connect to this edge are \begin{picture}(10,10)
\put(-5,-6){
\Scale[0.3]{\begin{tikzpicture}[line cap=round,line join=round,>=triangle 45,x=1.0cm,y=1.0cm]
\clip(1.12,-0.4) rectangle (2.76,2.1);
\fill[color=cccccc,fill=cccccc,fill opacity=1.0] (2,1.73) -- (1.5,0.87) -- (2,0) -- (2.5,0.87) -- cycle;
\draw [color=cccccc] (2,1.73)-- (1.5,0.87);
\draw [color=cccccc] (1.5,0.87)-- (2,0);
\draw [color=cccccc] (2,0)-- (2.5,0.87);
\draw [color=cccccc] (2.5,0.87)-- (2,1.73);
\end{tikzpicture}}}
\end{picture} and those from 
{\tt startrow} (at the stated shared edge).
\end{proof}
\noindent
Second, pieces of the puzzle not in a track or NWray are $0$-triangles (depicted white).

We correspond Young diagrams to $\{0,1\}$-sequences. Trace the $\{\leftarrow,\downarrow\}$-lattice path defined by the southern boundary
of $\lambda$ (as placed in the northwest corner of $k\times (n-k)$)
starting from the northeast corner of $k\times (n-k)$ towards the southeast corner of $k\times (n-k)$. Record each $\leftarrow$ step with ``$0$'' and each $\downarrow$ step with ``$1$''. 

We now convert $P$ into (we claim) an edge-labeled starred genomic tableau $T:=\phi(P)$ 
of shape $\nu/\lambda$ with content $\mu$.
The placement of the labels of family $i$ is governed by the 
decomposition (\ref{eqn:canonicalform})
of $\pi_i$. The initial sequence of $k$ 
\begin{picture}(10,10)
\put(-5,-1){
\Scale[0.3]{\begin{tikzpicture}[line cap=round,line join=round,>=triangle
45,x=1.0cm,y=1.0cm]
\clip(-0.21,-0.18) rectangle (1.72,1.04);
\fill[color=cccccc,fill=cccccc,fill opacity=1.0] (0,0) -- (0.5,0.87) --
(1.5,0.87) -- (1,0) -- cycle;
\end{tikzpicture}}}
\end{picture}\ 's
indicates the leftmost possible placement of box labels $i_{\mu_i},i_{\mu_i-1},\ldots,i_{\mu_i-k+1}$ (from right to left) in row $i$ of $T$. Continuing to read the sequence, one interprets:
\begin{itemize}
\item[(B.1)]
\begin{picture}(10,10)
\put(-5,-1){
\Scale[0.3]{\begin{tikzpicture}[line cap=round,line join=round,>=triangle
45,x=1.0cm,y=1.0cm]
\clip(-0.21,-0.18) rectangle (1.72,1.04);
\fill[color=cccccc,fill=cccccc,fill opacity=1.0] (0,0) -- (0.5,0.87) --
(1.5,0.87) -- (1,0) -- cycle;
\end{tikzpicture}}}
\end{picture}
\ $\leftrightarrow$ ``place (unstarred) box label
of next smaller gene''
\item[(B.2)]  %cappiece
\begin{picture}(10,10)
\put(-5,-2){
\Scale[0.3]{
\begin{tikzpicture}[line cap=round,line join=round,>=triangle
45,x=1.0cm,y=1.0cm]
\clip(0.67,-0.37) rectangle (2.27,1.04);
\fill[color=black,fill=black,fill opacity=1.0] (1.5,0.87) -- (1,0) --
(2,0) -- cycle;
\end{tikzpicture}}}
\end{picture}
$\leftrightarrow$``end placing box labels in current row''
\item[(B.3)] 
\begin{picture}(10,10)
\put(-5,-2){
\Scale[0.3]{
\begin{tikzpicture}[line cap=round,line join=round,>=triangle
45,x=1.0cm,y=1.0cm]
\clip(1.25,-0.23) rectangle (3.29,1.04);
\fill[color=cccccc,fill=cccccc,fill opacity=1.0] (1.5,0.87) -- (2.5,0.87) --
(3,0) -- (2,0) -- cycle;
\end{tikzpicture}}}
\end{picture}
\ $\leftrightarrow$``skip to the next column left''
\item[(B.4)] %greenpiece
\begin{picture}(10,10)
\put(-5,-6){
\Scale[0.3]{\begin{tikzpicture}[line cap=round,line join=round,>=triangle 45,x=1.0cm,y=1.0cm]
\clip(1.12,-0.4) rectangle (2.76,2.1);
\fill[color=qqffqq,fill=qqffqq,fill opacity=1.0] (2,1.73) -- (1.5,0.87) -- (2,0) -- (2.5,0.87) -- cycle;
\draw [color=qqffqq] (2,1.73)-- (1.5,0.87);
\draw [color=qqffqq] (1.5,0.87)-- (2,0);
\draw [color=qqffqq] (2,0)-- (2.5,0.87);
\draw [color=qqffqq] (2.5,0.87)-- (2,1.73);
\end{tikzpicture}}}
\end{picture}
$\leftrightarrow$``place lower edge label of the next smaller
gene''
\item[(B.5)] 
\begin{picture}(10,10)
\put(-5,-6){
\Scale[0.3]{\begin{tikzpicture}[line cap=round,line join=round,>=triangle
45,x=1.0cm,y=1.0cm]
\clip(-0.3,-3.89) rectangle (1.84,-1.8);
\fill[color=zzqqzz,fill=zzqqzz,fill opacity=1.0] (0,-2) -- (1,-2) --
(0.5,-2.87) -- cycle;
\fill[color=qqffqq,fill=qqffqq,fill opacity=1.0] (1,-2) -- (0.5,-2.87) --
(1,-3.73) -- (1.5,-2.87) -- cycle;
\end{tikzpicture}}}
\end{picture}
\ $\leftrightarrow$``place lower edge label
of the same gene last used''
\item[(B.6)] 
\begin{picture}(10,10)
\put(-5,-2){
\Scale[0.3]{\begin{tikzpicture}[line cap=round,line join=round,>=triangle
45,x=1.0cm,y=1.0cm]
\clip(0.72,0.58) rectangle (2.19,1.92);
\fill[color=black,fill=black,fill opacity=1.0] (1,1.73) -- (2,1.73) --
(1.5,0.87) -- cycle;
\end{tikzpicture}}}
\end{picture}
\ $\leftrightarrow$``go to next row''
\item[(B.7)] 
\begin{picture}(10,10)
\put(-5,-6){
\Scale[0.3]{
\begin{tikzpicture}[line cap=round,line join=round,>=triangle
45,x=1.0cm,y=1.0cm]
\clip(0.2,-3.98) rectangle (2.24,-1.63);
\fill[color=qqffqq,fill=qqffqq,fill opacity=1.0] (1,-2) -- (0.5,-2.87) --
(1,-3.73) -- (1.5,-2.87) -- cycle;
\fill[color=zzttqq,fill=zzttqq,fill opacity=1.0] (1,-2) -- (1.5,-2.87) --
(2,-2) -- cycle;
\draw [color=qqffqq] (1,-2)-- (0.5,-2.87);
\draw [color=qqffqq] (0.5,-2.87)-- (1,-3.73);
\draw [color=qqffqq] (1,-3.73)-- (1.5,-2.87);
\draw [color=qqffqq] (1.5,-2.87)-- (1,-2);
\draw [color=zzttqq] (1,-2)-- (1.5,-2.87);
\draw [color=zzttqq] (1.5,-2.87)-- (2,-2);
\draw [color=zzttqq] (2,-2)-- (1,-2);
\end{tikzpicture}}}
\end{picture}
\ $\leftrightarrow$``go to next row and place
$\star$-ed box label of the next smaller gene''
\item[(B.8)]
\begin{picture}(10,12)
\put(-5,-6){ 
\Scale[0.3]{\begin{tikzpicture}[line cap=round,line join=round,>=triangle
45,x=1.0cm,y=1.0cm]
\clip(-0.23,-3.93) rectangle (2.25,-1.86);
\fill[color=zzqqzz,fill=zzqqzz,fill opacity=1.0] (0,-2) -- (1,-2) --
(0.5,-2.87) -- cycle;
\fill[color=cccccc,fill=cccccc,fill opacity=1.0] (1,-2) -- (0.5,-2.87) --
(1.5,-2.87) -- (2,-2) -- cycle;
\fill[fill=black,fill opacity=1.0] (0.5,-2.87) -- (1.5,-2.87) -- (1,-3.73)
-- cycle;
\end{tikzpicture}}}
\end{picture}
\ $\leftrightarrow$``go to next row and place
(unstarred) box label of the same gene last used''
\item[(B.9)] 
\begin{picture}(10,12)
\put(-5,-6){
\Scale[0.3]{
\begin{tikzpicture}[line cap=round,line join=round,>=triangle
45,x=1.0cm,y=1.0cm]
\clip(-0.28,-3.92) rectangle (2.26,-1.9);
\fill[color=zzqqzz,fill=zzqqzz,fill opacity=1.0] (0,-2) -- (1,-2) --
(0.5,-2.87) -- cycle;
\fill[color=qqffqq,fill=qqffqq,fill opacity=1.0] (1,-2) -- (0.5,-2.87) --
(1,-3.73) -- (1.5,-2.87) -- cycle;
\fill[color=zzttqq,fill=zzttqq,fill opacity=1.0] (1,-2) -- (1.5,-2.87) --
(2,-2) -- cycle;
\end{tikzpicture}}}
\end{picture}
\ $\leftrightarrow$``go to next row and 
place $\star$-ed box label of the same gene last used''.
\end{itemize}

Applying $\phi$ to the puzzle $P$ of Figure~\ref{fig:bigex} gives the tableau $T$ of Figure~\ref{fig:thetab}. Here, $\lambda=
0^5 1 0^{10}1010$, corresponding to the inner shape $(12,2,1)$ (which is shaded in grey). Since $\mu=0^7 1 0^5 1 0^2 1 0^3$, the content
of $T$ is $(10,5,3)$. Finally, since $\nu=0^2 1 0^7 1 0^3 1 0^5$, the outer shape of
$T$ is $(15,8,5)$. As another example, $\phi$ connects the puzzles $P_2,P_3,P_5$ and $P_6$ 
of Section~\ref{sec:intro} respectively with the tableaux $T_2,T_3,T_5$ and $T_6$
of Example~\ref{exa:fourtableaux}.

\begin{figure}[h]
\begin{picture}(340,70)
\ytableausetup{boxsize=2.0em}
\put(-30,48){$\ytableaushort{ {*(lightgray)\blank} {*(lightgray)\blank} {*(lightgray)\blank} {*(lightgray)\blank} 
{*(lightgray)\blank} {*(lightgray)\blank}  {*(lightgray)\blank}  {*(lightgray)\blank} {*(lightgray)\blank} {*(lightgray)\blank} {*(lightgray)\blank} {*(lightgray)\blank} 
{1_8} {1_9} {1_{10}}, 
{*(lightgray)\blank} {*(lightgray)\blank} {1_2} {1_3} {2_2} {2_3} {2_4} {2_5},
{*(lightgray)\blank} {1_2^\star} {2_2} {3_2} {3_3}}$}
\put(222,43){$1_7$}
\put(198,43){$1_6$}
\put(148,43){$1_6$}
\put(100,43){$1_5$}
\put(76,43){$1_4$}
\put(51,19){$2_2$}
\put(-3,-5){$2_1 3_1$}
\put(-27,-5){$1_1 3_1$}
\end{picture}
\caption{The tableau $T := \phi(P)$ corresponding to the modified KV-puzzle $P$ of Figure~\ref{fig:bigex}.}
\label{fig:thetab}
\end{figure}

Conversely, given $T\in {\tt StarBallotGen}_\mu(\nu/\lambda)$, construct a word $\sigma_i$ using the
correspondences (B.1)--(B.9), for $1\leq i\leq k$. That is, read the occurrences (possibly zero) of family $i$ in $T$
from right to left and from the $i$th row down. 
(Note about (B.6) in the degenerate case
that there are no labels of family $i$ in the next row:
use \begin{picture}(10,10)
\put(-5,-2){
\Scale[0.3]{
\begin{tikzpicture}[line cap=round,line join=round,>=triangle
45,x=1.0cm,y=1.0cm]
\clip(0.67,-0.37) rectangle (2.27,1.04);
\fill[color=black,fill=black,fill opacity=1.0] (1.5,0.87) -- (1,0) --
(2,0) -- cycle;
\end{tikzpicture}}}
\end{picture}
after reading the leftmost column in $\nu/\lambda$ that does not have any labels of family $<i$.)

\begin{lemma}
\label{lemma:sigmadecomp}
Each $\sigma_i$  is of the form
(\ref{eqn:canonicalform}).
\end{lemma}
\begin{proof}
Since $T$ is semistandard, in any row, all box labels of family $i$ are contiguous
and strictly right of any (lower) edge labels of that family on that row. The lemma follows.
\end{proof}

We describe a claimed filling $P:=\psi(T)$
of $\Delta_{\lambda,\mu,\nu}$. There
are $k$ $1$'s on each side of $\Delta_{\lambda,\mu,\nu}$; to the $i$th $1$ from the left on the $\nu$-boundary of $\Delta_{\lambda,\mu,\nu}$,
place puzzle pieces in the order indicated by $\sigma_i$. That is attach the next (combination) piece using the northmost $\backslash$ edge on its west side, if it exists. Otherwise attach at 
the piece's unique southern edge. We attach at the unique --- or $\backslash$ edge of the thus far constructed track. Fill in the order $i=1,2,3,\ldots,k$. Now
stack \begin{picture}(10,10)
\put(-5,-6){
\Scale[0.3]{\begin{tikzpicture}[line cap=round,line join=round,>=triangle 45,x=1.0cm,y=1.0cm]
\clip(1.12,-0.4) rectangle (2.76,2.1);
\fill[color=cccccc,fill=cccccc,fill opacity=1.0] (2,1.73) -- (1.5,0.87) -- (2,0) -- (2.5,0.87) -- cycle;
\draw [color=cccccc] (2,1.73)-- (1.5,0.87);
\draw [color=cccccc] (1.5,0.87)-- (2,0);
\draw [color=cccccc] (2,0)-- (2.5,0.87);
\draw [color=cccccc] (2.5,0.87)-- (2,1.73);
\end{tikzpicture}}}
\end{picture}'s northwest of each \begin{picture}(10,10)
\put(-5,-2){
\Scale[0.3]{
\begin{tikzpicture}[line cap=round,line join=round,>=triangle
45,x=1.0cm,y=1.0cm]
\clip(0.67,-0.37) rectangle (2.27,1.04);
\fill[color=black,fill=black,fill opacity=1.0] (1.5,0.87) -- (1,0) --
(2,0) -- cycle;
\end{tikzpicture}}}
\end{picture} until (we claim) it reaches one of the pieces of (A.6)--(A.9) at the southmost 
$/$ edge, or the $\lambda$-boundary of $\Delta_{\lambda,\mu,\nu}$. Complete using white triangles.

Sections~\ref{sec:welldefphi}--\ref{sec:weight_pres} prove $\phi$ and $\psi$ are
well-defined and weight-preserving maps between 
\[{\mathcal P}:=\{\text{modified KV-puzzles of $\Delta_{\lambda,\mu,\nu}$}\} \text{ and } {\mathcal T}:={\tt StarBallotGen}_{\mu}(\nu/\lambda).\] 
Semistandardness (specifically (S.4)) implies that knowing the locations of labels 
of family $i$, and which labels are repeated or $\star$-ed, uniquely determines the gene(s) in each location.
The injectivity of $\phi$ and $\psi$ is easy from this. Moreover,
by construction (cf.~Lemma~\ref{lemma:sigmadecomp}), 
the two maps are mutually reversing. Thus, Theorem~\ref{thm:main} follows from Theorem~\ref{thm:reformulated}.\qed
%% Alternative sentence in case we have to send it out quick
%\excise{It is straightforward from the definitions to see
%that the two maps are
%well-defined and weight-preserving between 
% $\{\text{KV-puzzle fillings of $\Delta_{\lambda,\mu,\nu}$}\}$ and ${\tt StarBallotGen}(\nu/\lambda)$ of content $\mu$. Given this, Theorem~\ref{thm:main} follows from Theorem~\ref{thm:reformulated}.\qed}
%% End of alternative ending

\subsection{Well-definedness of $\phi:{\mathcal P}\to {\mathcal T}$}
\label{sec:welldefphi}
Let $P\in {\mathcal P}$ be a modified KV-puzzle for $\Delta_{\lambda,\mu,\nu}$.
For the track $\pi_i$, let
\begin{picture}(10,10)
\put(-5,-2){
\Scale[0.3]{
\begin{tikzpicture}[line cap=round,line join=round,>=triangle
45,x=1.0cm,y=1.0cm]
\clip(0.67,-0.37) rectangle (2.27,1.04);
\fill[color=black,fill=black,fill opacity=1.0] (1.5,0.87) -- (1,0) --
(2,0) -- cycle;
\end{tikzpicture}}}
\end{picture}$_{i,j}$ refer to the $j$th black triangle
seen along $\pi_i$ (as read from southwest to northeast). 
Let ${\mathbb S}$ denote any of the (combination) pieces that appear in {\tt startrow}. Similarly, we let ${\mathbb S}_{i,j}$ be the $j$th such piece on $\pi_i$.

Figure~\ref{fig:bigex} illustrates the ``ragged honeycomb'' structure of modified KV-puzzles.
To formalize this, first note by inspection that the $\pi_i$ do not intersect. Second we have:

\begin{claim}
\label{claim:AA}
There is a bijective correspondence between the $1$'s on the $\lambda$-boundary and the 
\begin{picture}(10,10)
\put(-5,-2){
\Scale[0.3]{
\begin{tikzpicture}[line cap=round,line join=round,>=triangle
45,x=1.0cm,y=1.0cm]
\clip(0.67,-0.37) rectangle (2.27,1.04);
\fill[color=black,fill=black,fill opacity=1.0] (1.5,0.87) -- (1,0) --
(2,0) -- cycle;
\end{tikzpicture}}}
\end{picture}'s in $\pi_1$. Specifically, the
$j$th $1$ on the $\lambda$-boundary is the terminus of
the NWray of 
\begin{picture}(10,10)
\put(-5,-2){
\Scale[0.3]{
\begin{tikzpicture}[line cap=round,line join=round,>=triangle
45,x=1.0cm,y=1.0cm]
\clip(0.67,-0.37) rectangle (2.27,1.04);
\fill[color=black,fill=black,fill opacity=1.0] (1.5,0.87) -- (1,0) --
(2,0) -- cycle;
\end{tikzpicture}}}
\end{picture}$_{1,j}$.
Similarly, there is a bijective correspondence between
\begin{picture}(10,10)
\put(-5,-2){
\Scale[0.3]{
\begin{tikzpicture}[line cap=round,line join=round,>=triangle
45,x=1.0cm,y=1.0cm]
\clip(0.67,-0.37) rectangle (2.27,1.04);
\fill[color=black,fill=black,fill opacity=1.0] (1.5,0.87) -- (1,0) --
(2,0) -- cycle;
\end{tikzpicture}}}
\end{picture}$_{i+1,j}$ and ${\mathbb S}_{i,j}$ in that the former's NWray terminates at the
southmost $/$ edge of the latter.
\end{claim}
\begin{proof}
Follows by combining Proposition~\ref{claim:canonical}
and Lemma~\ref{lemma:NWray}.
% The real worry is that the NWray could terminate at a black or S that is not part of
% the track above. This contradicts the 0's on the \mu border because the only way for
% this bad thing to occur is if one starts sprouting a partial track starting at this
% bad black or S piece, a contradiction. But I didn't feel like saying all this.
\end{proof}

Define ${\mathcal L}_i$ to be the {\bf left sequence} of $\pi_i$: Start at the 
southwest corner of $\Delta_{\lambda,\mu,\nu}$ and read the $\{\rightarrow, \nearrow\}$-lattice path that starts along the $\nu$-boundary and travels up the left boundary of
$\pi_i$. The $\{0,1\}$-sequence records the labels of the edges seen.
Similarly, define ${\mathcal R}_i$ to be the {\bf right sequence} of $\pi_i$ by
travelling up the right side of $\pi_i$ \emph{but 
only reading the $\rightarrow$ and $\nearrow$ edges.} (In Figure~\ref{fig:bigex},
${\mathcal L}_1=0^6 1 0^{10} 1 0 1 0(=\lambda)$
while ${\mathcal R}_1=0^{2}10^{11}1 0^210^2$.) 

In view of Claim~\ref{claim:AA}, the following is ``graphically'' clear by considering
the $n$ diagonal strips through $P$:
\begin{claim}
\label{claim:B}
${\mathcal L}_1=\lambda$, ${\mathcal L}_{i+1}={\mathcal R}_i$ for $1\leq i\leq k-1$,
and $R_k=\nu$.
\end{claim}

Let $T^{(i)}$ be the tableau after adding labels of family $1,2,\ldots,i$. We declare
$T^{(0)}$ to be the empty tableau of shape $\lambda/\lambda$. Let $\nu^{(i)}$ be the
outer shape of $T^{(i)}$ (interpreted as the $\{0,1\}$-sequence for its
lattice path).

\begin{claim}
\label{claim:C}
${\mathcal L}_i=\nu^{(i-1)}$ and ${\mathcal R}_i=\nu^{(i)}$.
\end{claim}
\begin{proof}
Both assertions follow by inspection of the correspondences (B.1)--(B.9). (Also the
second follows from the first, by Claim~\ref{claim:B}.)
\end{proof}

It is straightforward from
Claims~\ref{claim:B} and~\ref{claim:C} that $T=\phi(P)$
is semistandard in the sense of (S.1)--(S.4) of \cite{PeYo}. By Proposition~\ref{claim:canonical},
no label of $T$ is $\star$-ed unless it is the rightmost box label of its
family in a row ($>i$). Since labels of family $i$ are placed in the boxes of
row $i$ or below, no label of $T$ can be too high. Since
${\mathcal R}_{k}=\nu$, the shape of $T$ is $\nu/\lambda$. 

\begin{claim}
\label{claim:D}
$T$ has content $\mu$.
\end{claim}
\begin{proof}
Let $\beta$ be the content of $T$. Then $\beta_i$ is the number of (distinct) genes of family $i$ that appear in $T$, which, in terms of $P$, is the number
of 
\begin{picture}(10,10)
\put(-5,-1){
\Scale[0.3]{\begin{tikzpicture}[line cap=round,line join=round,>=triangle
45,x=1.0cm,y=1.0cm]
\clip(-0.21,-0.18) rectangle (1.72,1.04);
\fill[color=cccccc,fill=cccccc,fill opacity=1.0] (0,0) -- (0.5,0.87) --
(1.5,0.87) -- (1,0) -- cycle;
\end{tikzpicture}}}
\end{picture}
and
\begin{picture}(10,10)
\put(-5,-6){
\Scale[0.3]{\begin{tikzpicture}[line cap=round,line join=round,>=triangle 45,x=1.0cm,y=1.0cm]
\clip(1.12,-0.4) rectangle (2.76,2.1);
\fill[color=qqffqq,fill=qqffqq,fill opacity=1.0] (2,1.73) -- (1.5,0.87) -- (2,0) -- (2.5,0.87) -- cycle;
\draw [color=qqffqq] (2,1.73)-- (1.5,0.87);
\draw [color=qqffqq] (1.5,0.87)-- (2,0);
\draw [color=qqffqq] (2,0)-- (2.5,0.87);
\draw [color=qqffqq] (2.5,0.87)-- (2,1.73);
\end{tikzpicture}}}
\end{picture} \ 
in $\pi_i$ minus the number of purple KV-pieces
\begin{picture}(10,10)
\put(-5,-2){
\Scale[0.3]{\begin{tikzpicture}[line cap=round,line join=round,>=triangle
45,x=1.0cm,y=1.0cm]
\clip(0.72,0) rectangle (2.19,1.92);
\fill[color=zzqqzz,fill=zzqqzz,fill opacity=1.0] (1,1.73) -- (2,1.73) --
(1.5,0.87) -- cycle;
\draw[color=zzqqzz] (1.5,0.87) -- (2, 0);
\end{tikzpicture}}}
\end{picture}
in $\pi_i$. Thus the vertical height $h_i$ 
of $\pi_i$ (at its right end\emph{point}) is 
$\beta_i+\#$
\begin{picture}(10,10)
\put(-5,-2){
\Scale[0.3]{
\begin{tikzpicture}[line cap=round,line join=round,>=triangle
45,x=1.0cm,y=1.0cm]
\clip(0.67,-0.37) rectangle (2.27,1.04);
\fill[color=black,fill=black,fill opacity=1.0] (1.5,0.87) -- (1,0) --
(2,0) -- cycle;
\end{tikzpicture}}}
\end{picture}. However, $h_i$ equals
the number of line segments strictly below the $i$th $1$ on the $\mu$-boundary; i.e.,
$h_i=n-i-(n-k-\mu_i)=(k-i)+\mu_i$. By Claims~\ref{claim:AA} and~\ref{claim:canonical},
$\#$
\begin{picture}(10,10)
\put(-5,-2){
\Scale[0.3]{
\begin{tikzpicture}[line cap=round,line join=round,>=triangle
45,x=1.0cm,y=1.0cm]
\clip(0.67,-0.37) rectangle (2.27,1.04);
\fill[color=black,fill=black,fill opacity=1.0] (1.5,0.87) -- (1,0) --
(2,0) -- cycle;
\end{tikzpicture}}}
\end{picture}
$=(k-i)$, hence $\beta=\mu$, as desired.
\end{proof}
Finally,
\begin{claim}
\label{claim:ballot}
$T$ is ballot.
\end{claim}
\begin{proof}
The {\bf height} of a (combination) piece is the distance of any northernmost point to the
$\nu$-boundary as measured along any (anti)diagonal.
The height $h$ of 
\begin{picture}(10,10)
\put(-5,-2){
\Scale[0.3]{
\begin{tikzpicture}[line cap=round,line join=round,>=triangle
45,x=1.0cm,y=1.0cm]
\clip(0.67,-0.37) rectangle (2.27,1.04);
\fill[color=black,fill=black,fill opacity=1.0] (1.5,0.87) -- (1,0) --
(2,0) -- cycle;
\end{tikzpicture}}}
\end{picture}$_{i+1,j}$ 
equals the number of 
\begin{picture}(10,10)
\put(-5,-1){
\Scale[0.3]{\begin{tikzpicture}[line cap=round,line join=round,>=triangle
45,x=1.0cm,y=1.0cm]
\clip(-0.21,-0.18) rectangle (1.72,1.04);
\fill[color=cccccc,fill=cccccc,fill opacity=1.0] (0,0) -- (0.5,0.87) --
(1.5,0.87) -- (1,0) -- cycle;
\end{tikzpicture}}}
\end{picture}\,'s,
\begin{picture}(10,10)
\put(-5,-2){
\Scale[0.3]{
\begin{tikzpicture}[line cap=round,line join=round,>=triangle
45,x=1.0cm,y=1.0cm]
\clip(0.67,-0.37) rectangle (2.27,1.04);
\fill[color=black,fill=black,fill opacity=1.0] (1.5,0.87) -- (1,0) --
(2,0) -- cycle;
\end{tikzpicture}}}
\end{picture}'s \
and
\begin{picture}(10,10)
\put(-5,-6){
\Scale[0.3]{\begin{tikzpicture}[line cap=round,line join=round,>=triangle 45,x=1.0cm,y=1.0cm]
\clip(1.12,-0.4) rectangle (2.76,2.1);
\fill[color=qqffqq,fill=qqffqq,fill opacity=1.0] (2,1.73) -- (1.5,0.87) -- (2,0) -- (2.5,0.87) -- cycle;
\draw [color=qqffqq] (2,1.73)-- (1.5,0.87);
\draw [color=qqffqq] (1.5,0.87)-- (2,0);
\draw [color=qqffqq] (2,0)-- (2.5,0.87);
\draw [color=qqffqq] (2.5,0.87)-- (2,1.73);
\end{tikzpicture}}}
\end{picture}'s
\ that appear weakly before 
\begin{picture}(10,10)
\put(-5,-2){
\Scale[0.3]{
\begin{tikzpicture}[line cap=round,line join=round,>=triangle
45,x=1.0cm,y=1.0cm]
\clip(0.67,-0.37) rectangle (2.27,1.04);
\fill[color=black,fill=black,fill opacity=1.0] (1.5,0.87) -- (1,0) --
(2,0) -- cycle;
\end{tikzpicture}}}
\end{picture}$_{i+1,j}$ in $\pi_{i+1}$ minus the number of 
\begin{picture}(10,10)
\put(-5,-2){
\Scale[0.3]{\begin{tikzpicture}[line cap=round,line join=round,>=triangle
45,x=1.0cm,y=1.0cm]
\clip(0.72,0) rectangle (2.19,1.92);
\fill[color=zzqqzz,fill=zzqqzz,fill opacity=1.0] (1,1.73) -- (2,1.73) --
(1.5,0.87) -- cycle;
\draw[color=zzqqzz] (1.5,0.87) -- (2, 0);
\end{tikzpicture}}}
\end{picture}'s
before \begin{picture}(10,10)
\put(-5,-2){
\Scale[0.3]{
\begin{tikzpicture}[line cap=round,line join=round,>=triangle
45,x=1.0cm,y=1.0cm]
\clip(0.67,-0.37) rectangle (2.27,1.04);
\fill[color=black,fill=black,fill opacity=1.0] (1.5,0.87) -- (1,0) --
(2,0) -- cycle;
\end{tikzpicture}}}
\end{picture}$_{i+1,j}$ in $\pi_{i+1}$.
There are exactly $j$ such 
\begin{picture}(10,10)
\put(-5,-2){
\Scale[0.3]{
\begin{tikzpicture}[line cap=round,line join=round,>=triangle
45,x=1.0cm,y=1.0cm]
\clip(0.67,-0.37) rectangle (2.27,1.04);
\fill[color=black,fill=black,fill opacity=1.0] (1.5,0.87) -- (1,0) --
(2,0) -- cycle;
\end{tikzpicture}}}
\end{picture}'s, while the number of 
\begin{picture}(10,10)
\put(-5,-1){
\Scale[0.3]{\begin{tikzpicture}[line cap=round,line join=round,>=triangle
45,x=1.0cm,y=1.0cm]
\clip(-0.21,-0.18) rectangle (1.72,1.04);
\fill[color=cccccc,fill=cccccc,fill opacity=1.0] (0,0) -- (0.5,0.87) --
(1.5,0.87) -- (1,0) -- cycle;
\end{tikzpicture}}}
\end{picture}\,'s and
\begin{picture}(10,10)
\put(-5,-6){
\Scale[0.3]{\begin{tikzpicture}[line cap=round,line join=round,>=triangle 45,x=1.0cm,y=1.0cm]
\clip(1.12,-0.4) rectangle (2.76,2.1);
\fill[color=qqffqq,fill=qqffqq,fill opacity=1.0] (2,1.73) -- (1.5,0.87) -- (2,0) -- (2.5,0.87) -- cycle;
\draw [color=qqffqq] (2,1.73)-- (1.5,0.87);
\draw [color=qqffqq] (1.5,0.87)-- (2,0);
\draw [color=qqffqq] (2,0)-- (2.5,0.87);
\draw [color=qqffqq] (2.5,0.87)-- (2,1.73);
\end{tikzpicture}}}
\end{picture}'s is the number of labels used and 
the number of 
\begin{picture}(10,10)
\put(-5,-2){
\Scale[0.3]{\begin{tikzpicture}[line cap=round,line join=round,>=triangle
45,x=1.0cm,y=1.0cm]
\clip(0.72,0) rectangle (2.19,1.92);
\fill[color=zzqqzz,fill=zzqqzz,fill opacity=1.0] (1,1.73) -- (2,1.73) --
(1.5,0.87) -- cycle;
\draw[color=zzqqzz] (1.5,0.87) -- (2, 0);
\end{tikzpicture}}}
\end{picture}'s is the number of these labels that are repeats. That is $h=j+(\#\text{distinct genes of family
$i+1$ in row $j+1$ and above})$ where we
do \emph{not} include labels on the lower
edges of row $j+1$. Similarly,
the height $h'$ of
${\mathbb S}_{i,j}$ is given by
$h'=j+(\#\text{distinct genes of family $i$ in row $j$ and above})$ where we include labels on the 
lower edges of row $j$.  By Claim~\ref{claim:AA}, $h'-h\geq 0$ and so ballotness follows.
\end{proof}

\subsection{Well-definedness of $\psi:{\mathcal T}\to {\mathcal P}$}
\label{sec:welldefpsi}

Let $T\in {\mathcal T}$ be a starred ballot genomic tableau of shape $\nu/\lambda$ and content $\mu$. Let $P=\psi(T)$. Let $\pi_i$ be the track associated to $\sigma_i$. As in Section~\ref{sec:welldefphi}, we define
the $\{0,1\}$-sequences ${\mathcal L}_i$ and ${\mathcal R}_i$ associated to 
$\pi_i$. Here, $T^{(i)}$ is \emph{defined} as the subtableau of $T$ using the labels of 
family $1,2,\ldots,i$. Hence $T^{(0)}$ is the empty tableau of shape $\lambda/\lambda$. Let $\nu^{(i)}$ be the outer shape of $T^{(i)}$.

\begin{claim}
[cf.~Claim~\ref{claim:C}]
\label{claim:CCCC}
${\mathcal L}_i=\nu^{(i-1)}$ and ${\mathcal R}_i=\nu^{(i)}$.
\end{claim}
\begin{proof}
By inspection of the correspondences (B.1)-(B.9). 
\end{proof}

By the lattice path definition, each $\nu^{(j)}$ is a length $n$ sequence. So $\pi_i$ is a track that (by definition) starts at the south border
of $\Delta$ and terminates at the east border of $\Delta$. 
Also, define 
\begin{picture}(10,10)
\put(-5,-2){
\Scale[0.3]{
\begin{tikzpicture}[line cap=round,line join=round,>=triangle
45,x=1.0cm,y=1.0cm]
\clip(0.67,-0.37) rectangle (2.27,1.04);
\fill[color=black,fill=black,fill opacity=1.0] (1.5,0.87) -- (1,0) --
(2,0) -- cycle;
\end{tikzpicture}}}
\end{picture}$_{i,j}$ and ${\mathbb S}_{i,j}$ as before. 

\begin{claim}
\label{claim:E'}
${\mathbb S}_{i,j}$ and
\begin{picture}(10,10)
\put(-5,-2){
\Scale[0.3]{
\begin{tikzpicture}[line cap=round,line join=round,>=triangle
45,x=1.0cm,y=1.0cm]
\clip(0.67,-0.37) rectangle (2.27,1.04);
\fill[color=black,fill=black,fill opacity=1.0] (1.5,0.87) -- (1,0) --
(2,0) -- cycle;
\end{tikzpicture}}}
\end{picture}$_{i+1,j}$ share a 
diagonal with the former strictly northwest of the latter.
\end{claim}
\begin{proof}
The $1$'s in ${\mathcal L}_{i+1}$ result solely from
the \begin{picture}(10,10)
\put(-5,-2){
\Scale[0.3]{
\begin{tikzpicture}[line cap=round,line join=round,>=triangle
45,x=1.0cm,y=1.0cm]
\clip(0.67,-0.37) rectangle (2.27,1.04);
\fill[color=black,fill=black,fill opacity=1.0] (1.5,0.87) -- (1,0) --
(2,0) -- cycle;
\end{tikzpicture}}}
\end{picture}'s
 appearing in $\pi_{i+1}$ while the $1$'s appearing in 
${\mathcal R}_i$ result solely from the ${\mathbb S}$ (combination) pieces.
Thus, that the pieces share a diagonal
follows from Claim~\ref{claim:CCCC}. For the ``northwest'' assertion, repeat
Claim~\ref{claim:ballot}'s argument but reverse the logic of the final sentence: since
by assumption $T$ is ballot, it follows that $h'\geq h$. 
\end{proof}

Since Claims~\ref{claim:CCCC} and~\ref{claim:E'} combine to imply that the $\pi_i$ are non-intersecting, attaching NWrays to each \begin{picture}(10,10)
\put(-5,-2){
\Scale[0.3]{
\begin{tikzpicture}[line cap=round,line join=round,>=triangle
45,x=1.0cm,y=1.0cm]
\clip(0.67,-0.37) rectangle (2.27,1.04);
\fill[color=black,fill=black,fill opacity=1.0] (1.5,0.87) -- (1,0) --
(2,0) -- cycle;
\end{tikzpicture}}}
\end{picture}
and filling with white $0$-triangles as prescribed, we have a filling $P$ of $\Delta_{{\widetilde \lambda},{\widetilde \mu},\nu}$ satisfying the modified KV-puzzle rule. It remains to check the $\lambda$- and $\mu$-boundaries.

\begin{claim}
${\widetilde \lambda}=\lambda$.
\end{claim}
\begin{proof}
Graphically, ${\widetilde \lambda}={\mathcal L}_1$. On the other hand, 
by Claim~\ref{claim:CCCC}, we know that ${\mathcal L}_1=\lambda$. 
\end{proof}

\begin{claim}
${\widetilde \mu}=\mu$.
\end{claim}
\begin{proof}
This is given by reversing the logic of the proof of Claim~\ref{claim:D};
here we are given the content of $T$ and are determining the
heights of the tracks $\pi_i$.
\end{proof}

\subsection{Weight-preservation}
\label{sec:weight_pres}
We wish to show:
\begin{claim}
$\phi$ is weight-preserving, i.e., ${\rm wt}(P)={\widehat{\rm wt}}(T)$.
\end{claim}
\begin{proof}
The $\pm 1$ sign associated to $P$ and $T$ is the same
since each usage of a KV-piece in $P$ corresponds to a $\star$-ed label or a repetition of a
gene in $T$.

Now consider the weight $1-\frac{t_a}{t_b}$ assigned to an equivariant piece $p$ in $P$. Here $a$ is the ordinal (counted from the right) of the line segment $s$ on the $\nu$-boundary hit by the diagonal ``right leg'' emanating from $p$. Then $b$ equals $a+h-1$ where $h$ is the height of the piece $p$.
Suppose $p$ lies in track $\pi_i$, and corresponds either to $i_j$ on the lower edge of box $\x$ in row $r$ or to $i_j^\star \in \x$ in row $r$. Consider the edge $e$ on the left boundary of $\pi_i$ that is on the same diagonal as $s$. If $p$ is not attached to the first KV-piece, so it corresponds to an edge label, then $e$'s index from the right in the string $\mathcal{L}_i$ equals ${\sf Man}(\x)$. Otherwise $e$'s index from the right in the string $\mathcal{L}_i$ equals ${\sf Man}(\x) + 1$.

Note that $h$ equals the number of
\begin{picture}(10,10)
\put(-5,-1){
\Scale[0.3]{\begin{tikzpicture}[line cap=round,line join=round,>=triangle
45,x=1.0cm,y=1.0cm]
\clip(-0.21,-0.18) rectangle (1.72,1.04);
\fill[color=cccccc,fill=cccccc,fill opacity=1.0] (0,0) -- (0.5,0.87) --
(1.5,0.87) -- (1,0) -- cycle;
\end{tikzpicture}}}
\end{picture}\,'s,
\begin{picture}(10,10)
\put(-5,-2){
\Scale[0.3]{
\begin{tikzpicture}[line cap=round,line join=round,>=triangle
45,x=1.0cm,y=1.0cm]
\clip(0.67,-0.37) rectangle (2.27,1.04);
\fill[color=black,fill=black,fill opacity=1.0] (1.5,0.87) -- (1,0) --
(2,0) -- cycle;
\end{tikzpicture}}}
\end{picture}'s and
\begin{picture}(10,10)
\put(-5,-6){
\Scale[0.3]{\begin{tikzpicture}[line cap=round,line join=round,>=triangle 45,x=1.0cm,y=1.0cm]
\clip(1.12,-0.4) rectangle (2.76,2.1);
\fill[color=qqffqq,fill=qqffqq,fill opacity=1.0] (2,1.73) -- (1.5,0.87) -- (2,0) -- (2.5,0.87) -- cycle;
\draw [color=qqffqq] (2,1.73)-- (1.5,0.87);
\draw [color=qqffqq] (1.5,0.87)-- (2,0);
\draw [color=qqffqq] (2,0)-- (2.5,0.87);
\draw [color=qqffqq] (2.5,0.87)-- (2,1.73);
\end{tikzpicture}}}
\end{picture}'s appearing weakly before
$p$ in $\pi_i$ minus the number of \begin{picture}(10,10)
\put(-5,-2){
\Scale[0.3]{\begin{tikzpicture}[line cap=round,line join=round,>=triangle
45,x=1.0cm,y=1.0cm]
\clip(0.72,0) rectangle (2.19,1.92);
\fill[color=zzqqzz,fill=zzqqzz,fill opacity=1.0] (1,1.73) -- (2,1.73) --
(1.5,0.87) -- cycle;
\draw[color=zzqqzz] (1.5,0.87) -- (2, 0);
\end{tikzpicture}}}
\end{picture}'s appearing before $p$ in $\pi_i$. The number of such
\begin{picture}(10,10)
\put(-5,-2){
\Scale[0.3]{
\begin{tikzpicture}[line cap=round,line join=round,>=triangle
45,x=1.0cm,y=1.0cm]
\clip(0.67,-0.37) rectangle (2.27,1.04);
\fill[color=black,fill=black,fill opacity=1.0] (1.5,0.87) -- (1,0) --
(2,0) -- cycle;
\end{tikzpicture}}}
\end{picture}'s equals $1 + r-i$ if $p$ corresponds to an edge label and equals $r-i$ if $p$ corresponds to a starred label. The number of such \begin{picture}(10,10)
\put(-5,-1){
\Scale[0.3]{\begin{tikzpicture}[line cap=round,line join=round,>=triangle
45,x=1.0cm,y=1.0cm]
\clip(-0.21,-0.18) rectangle (1.72,1.04);
\fill[color=cccccc,fill=cccccc,fill opacity=1.0] (0,0) -- (0.5,0.87) --
(1.5,0.87) -- (1,0) -- cycle;
\end{tikzpicture}}}
\end{picture}\,'s and \begin{picture}(10,10)
\put(-5,-6){
\Scale[0.3]{\begin{tikzpicture}[line cap=round,line join=round,>=triangle 45,x=1.0cm,y=1.0cm]
\clip(1.12,-0.4) rectangle (2.76,2.1);
\fill[color=qqffqq,fill=qqffqq,fill opacity=1.0] (2,1.73) -- (1.5,0.87) -- (2,0) -- (2.5,0.87) -- cycle;
\draw [color=qqffqq] (2,1.73)-- (1.5,0.87);
\draw [color=qqffqq] (1.5,0.87)-- (2,0);
\draw [color=qqffqq] (2,0)-- (2.5,0.87);
\draw [color=qqffqq] (2.5,0.87)-- (2,1.73);
\end{tikzpicture}}}
\end{picture}'s minus the number of such \begin{picture}(10,10)
\put(-5,-2){
\Scale[0.3]{\begin{tikzpicture}[line cap=round,line join=round,>=triangle
45,x=1.0cm,y=1.0cm]
\clip(0.72,0) rectangle (2.19,1.92);
\fill[color=zzqqzz,fill=zzqqzz,fill opacity=1.0] (1,1.73) -- (2,1.73) --
(1.5,0.87) -- cycle;
\draw[color=zzqqzz] (1.5,0.87) -- (2, 0);
\end{tikzpicture}}}
\end{picture}'s equals $\mu_i -j + 1$. 
Weight preservation follows.
\end{proof}

\section*{Acknowledgments}
We thank Hugh Thomas for enlightening conversations during our work on \cite{PeYo} and also thank Ravi
Vakil for helpful correspondence.
OP was supported by an Illinois Distinguished Fellowship, an NSF Graduate Research Fellowship and NSF MCTP
grant DMS 0838434.
AY was supported by NSF grants.


\begin{thebibliography}{9999999999}
\bibitem[AnGrMi11]{Anderson.Griffeth.Miller} D.~Anderson, S.~Griffeth and
E.~Miller, \emph{Positivity and Kleiman transversality in equivariant K-theory of homogeneous spaces},
J.~Eur. Math. Soc., {\bf 13} (2011), 57--84.
\bibitem[Bu15]{Buch:quantum} A.~Buch, \emph{Mutations of puzzles and equivariant cohomology of two-step flag varieties}, Ann.\ of Math.~{\bf 182} (2015), 173--220.
\bibitem[BKPT13]{Buch.Kresch.Purbhoo.Tamvakis} A.~Buch, A.~Kresch, K.~Purbhoo and H.~Tamvakis, \emph{The puzzle conjecture for the cohomology of two-step flag manifolds}, preprint, 2013. \textsf{arXiv:1401.1725}
\bibitem[BuKrTa03]{Buch.Kresch.Tamvakis} A.~Buch, A.~Kresch and H.~Tamvakis,
\emph{Gromov-Witten invariants on Grassmannians}, J.~Amer.~Math.~Soc., 
{\bf 16} (2003), 901--915.
\bibitem[CoVa05]{Coskun.Vakil} I.~Co\c{s}kun and R.~Vakil,
\emph{Geometric positivity in the cohomology of homogeneous spaces and generalized Schubert calculus},
in ``Algebraic Geometry --- Seattle 2005'' Part 1, 77--124, Proc. Sympos. Pure Math., 80, Amer. Math. Soc., Providence, RI, 2009.
\bibitem[FuLa94]{fulton.lascoux} W.~Fulton and A.~Lascoux, \emph{A Pieri formula in the Grothendieck ring of a flag bundle}, Duke Math. J., {\bf 76}(3) (1994), 711--729.
\bibitem[Kn10]{Knutson:positroid} A.~Knutson, \emph{Puzzles, positroid varieties, and equivariant
$K$-theory of Grassmannians}, preprint, 2010. \textsf{arXiv:1008.4302}
\bibitem[KnPu11]{Knutson.Purbhoo} A.~Knutson and K.~Purbhoo, \emph{Product and puzzle formulae for $GL(n)$ Belkale-Kumar coefficients}, Electron. J.~Combin., {\bf 18}(1) (2011), P76.
\bibitem[KnTa03]{Knutson.Tao} A.~Knutson and T.~Tao, \emph{Puzzles and
(equivariant) cohomology of Grassmannians}, Duke~Math.~J. {\bf 119}(2) (2003), 221--260.
\bibitem[KnTaWo04]{KTW} A.~Knutson, T.~Tao and C.~Woodward, \emph{The honeycomb model of $GL_n(\complexes)$ tensor products II: puzzles determine facets of the Littlewood-Richardson cone}, J. Amer. Math. Soc. {\bf 17} (2004), 19--48.
\bibitem[Kr10]{Kreiman} V.~Kreiman, \emph{Equivariant Littlewood-Richardson skew tableaux},
Trans. Amer. Math. Soc. {\bf 362}(2010), 2589--2617.
\bibitem[LaSc82]{lascoux.schuetzenberger} A.~Lascoux and M.-P.~Sch\"utzenberger, \emph{Structure de Hopf de l'anneau de cohomologie et de l'anneau de Grothendieck d'une vari\'et\'e de drapeaux}, C. R. Acad. Sci. Paris, {\bf 295} (1982), 629--633.
\bibitem[PeYo15]{PeYo} O.~Pechenik and A.~Yong, \emph{Equivariant $K$-theory of Grassmannians}, preprint 2015. \textsf{arXiv:1506.01992}
\bibitem[Pu08]{purbhoo} K.~Purbhoo, \emph{Puzzles, tableaux, and mosaics}, J. Algebraic Combin., {\bf 28} (2008), 461--480.
\bibitem[ThYo13]{Thomas.Yong:H_T} H.~Thomas and A.~Yong, \emph{Equivariant Schubert calculus and jeu de taquin}, Ann. Inst. Fourier (Grenoble), to appear, 2013.
\bibitem[Va06]{Vakil:annals} R.~Vakil, \emph{A geometric Littlewood-Richardson rule}, Ann. of Math.,
{\bf 164} (2006), 371--422.
\end{thebibliography}
\end{document}